\documentclass[10pt]{article}
\newcommand{\documentdate}{\today}

\usepackage{a4wide,latexsym,amsmath,varioref,graphicx,xcolor,amssymb,diagbox,comment}
\usepackage{pifont}
\usepackage{tikz}
\usepackage{pgfplots}
\usepackage{multirow}
\usepackage{booktabs}
\usepackage{makecell,tablefootnote}
\usepackage{color}
\definecolor{brightpink}{rgb}{1.0, 0.0, 0.5}

\title{bAdag: an adaptive block coordinate gradient method for smooth nonconvex functions }

\author{ Giovanni Seraghiti\thanks{University of Mons, Rue de Houdain 9, 7000 Mons, Belgium. Email: {\tt giovanni.seraghiti@umons.ac.be}}
 \thanks{Dipartimento di Ingegneria Industriale, Universit\`a degli Studi di Firenze, Viale Morgagni 40/44, 50134 Firenze, Italia. GS acknowledge the support by the European Union (ERC consolidator, eLinoR, no 101085607). The work of GS\ was partially supported by INdAM-GNCS through Progetti di Ricerca.}
}

\newcommand{\beqn}[1]{\begin{equation}\label{#1}}
\newcommand{\eeqn}{\end{equation}}
\newcommand{\req}[1]{(\ref{#1})}

\newtheorem{theorem}{Theorem}[section]
\newtheorem{lemma}[theorem]{Lemma}

\newtheorem{corollary}[theorem]{Corollary}

\newcounter{algo}[section]
\renewcommand{\thealgo}{\thesection.\arabic{algo}}
\newcommand{\llem}[2]{\vspace{\baselineskip} 
\noindent\framebox[\textwidth]{\parbox{0.95\textwidth}{
\begin{lemma} \label{#1} \rm #2 \end{lemma} } } \vspace{\baselineskip} }

\newcommand{\algo}[3]{\refstepcounter{algo}
\begin{center}\begin{figure}[htbp]
\framebox[\textwidth]{
\parbox{0.95\textwidth} {\vspace{\topsep}
{\bf Algorithm \thealgo : #2}\label{#1}\\
\vspace*{-\topsep} \mbox{ }\\
{#3} \vspace{\topsep} }}
\end{figure}\end{center}}
\newcommand{\bpr}{{\bf Proof.} \hspace{1.5mm}}
\newcommand{\epr}{\hfill $\Box$ \vspace*{1em}}
\newcommand{\proof}[1]{
\begin{list}{}{
\setlength{\topsep}{0.0pt}
\setlength{\partopsep}{0.0pt}
\setlength{\leftmargin}{0.025\textwidth}
\setlength{\rightmargin}{0.5\leftmargin}
\setlength{\labelwidth}{0.5\leftmargin}
\setlength{\labelsep}{0.25\leftmargin}}
\item \bpr #1 \epr \noindent
\end{list}}
\newcommand{\lthm}[2]{\vspace{\baselineskip} 
\noindent\framebox[\textwidth]{\parbox{0.95\textwidth}{
\begin{theorem} \label{#1} \rm #2 \end{theorem} } } \vspace{\baselineskip} }

\newcommand{\iiz}[1]{\{ 0, \ldots, #1 \}}

\newcommand{\calA}{{\cal A}} 
\newcommand{\calD}{{\cal D}}

\newcommand{\calR}{{\cal R}} 
\newcommand{\calO}{{\cal O}}

\newcommand{\E}{\mathbb{E}}
\renewcommand{\Re}{\hbox{I\hskip -2pt R}}

\newcommand{\bigfrac}[2]{\frac{\displaystyle #1}{\displaystyle #2}}

\newcommand{\eqdef}{\stackrel{\rm def}{=}}

\newcommand{\tal}[1]{{\normalsize {\sf #1}}}

\newcommand{\flow}{f_{\rm low}}

\DeclareMathOperator*{\average}{average}
\newcommand{\sign}{\text{sign}}
\newcommand{\argmax}{\text{argmax}}

\newcommand{\name}{\text{\sf bAdag}}
\newcommand{\ename}{\text{\sf ebAdag}}
\newcommand{\namec}{\text{box-constrained \sf bAdag}}
\newcommand{\names}{\text{\sf prunAdag}}

\definecolor{greeng}{rgb}{0.0, 0.5, 0.4}
\newcommand{\gs}[1]{{\color{greeng} (\textbf{GS:} #1)}}
\newcommand{\cmark}{\ding{51}} 
\newcommand{\xmark}{\ding{55}} 
\date{\documentdate}

\begin{document}

\maketitle

\begin{abstract}
A new Block Coordinate Gradient (BCG) method, dubbed \name, for smooth, nonconvex minimization problem is proposed; it falls in the class of Objective Function Free Optimization (OFFO) methods, and it is based on the AdaGrad algorithm. At each iteration, our method computes an adaptive step size based on the cumulative sum of block gradients, instead of full gradients as in AdaGrad-type methods. We prove ergodic, sublinear convergence rates for the \name\ algorithm when minimizing a smooth, possibly nonconvex objective under the (block) Lipschitz continuity assumption on the gradient. Our theory covers three widely popular block selection strategies: the Cyclic (C) rule, Uniform Random selection (UR), and the greedy Gauss-Southwell (GS) rule. We also extend our algorithm and its convergence theory to box-constrained smooth functions. We validate the proposed algorithms through synthetic and real-world experiments.
\end{abstract}


\section{Introduction}

We are interested in solving the following unconstrained optimization problem
\begin{equation}
\min_{x \in \mathbb{R}^N} f(x),
    \label{eq:problem}
\end{equation}
with variables $x = (x_1, \dots, x_N) \in \mathbb{R}^N$, and where $f$ is a continuously differentiable, possibly nonconvex function. 
In large-scale optimization, a common iterative approach is Block Coordinate Descent (BCD), where, at iteration $k$, the objective function is minimized with respect to a subset of variables $\{x_i\}_{i \in b_{\ell_k}}$ and $b_{\ell_k} \subseteq \{1\dots N\}$, while the others remain fixed. Among the BCD schemes, a widely-used approach in smooth settings is the Block Coordinate Gradient (BCG) method, which uses the following update:
\begin{equation}
    x^{k+1}=x^k-\alpha_k g_k^{(\ell_k)},
    \label{eq:block_grad_step}
\end{equation}
where $x_k$ is the $k$-th iterate, $g_k\eqdef \nabla f(x_k)$ denotes the gradient of $f$ at $x_k$, and $g_k^{(\ell_k)}$ is the vector containing the partial derivatives corresponding to the indices in $b_{\ell_k}$ and zero elsewhere, and $\alpha_k$ is the step size. Due to its simplicity, this approach has been applied in a variety of contexts, especially in applications where the cost of one iteration of gradient descent is at least as expensive as updating all the components in the BCG scheme. Such objective functions are widely popular and include quadratic functions and logistic loss~\cite{nesterov2012efficiency,Nutini_faster,peng2016coordinate}. 

BCG methods are often classified based on the criteria for selecting the block to optimize at each iteration. Among the widely-known criteria, we identify three main categories:
\begin{enumerate}
    \item greedy strategies, such as the \textit{Gauss-Southwell} (GS) rule, which selects the block corresponding to the largest gradient norm;
    \item random strategies, where the block is chosen according to some probability distribution, for example, \textit{Uniformly at Random} (UR); 
    \item \textit{Cyclic} (C) rules, updating the blocks of variables cyclically.
\end{enumerate}  We further discuss the different types of BCG methods and their convergence in Section~\ref{sec:rel_work}. 

While the criteria for selecting the block to optimize at each iteration have been extensively investigated, surprisingly, much less effort has been devoted to the choice of the step $\alpha_k$ and how it affects the convergence of the BCG method. To the best of our knowledge, the majority of the contributions on BCG either consider constant step sizes or employ line-search strategies, which require several function evaluations at each iteration. However, gradient methods directly evaluating the objective function exhibit slow convergence when the function is expensive to compute or suffer from poor accuracy in noisy scenarios. A possible alternative are Objective Function Free Optimization (OFFO) methods that never evaluate the objective function. One of the reasons of their popularity is that, when dealing with noisy functions, they are more robust than gradient methods requiring function evaluations~\cite{gratton2024complexity}. Adaptive gradient methods are relevant examples of OFFO algorithms that rely on the history of the past gradient values to compute the step size, or the step direction, or both at every iteration. Prominent examples are AdaGrad~\cite{duchi2011adaptive,mcmahan2010adaptive}, Adam~\cite{kingma2014adam}, RMSprop~\cite{tieleman2012lecture}, ADADELTA~\cite{zeiler2012adadelta}, Shampoo~\cite{gupta2018shampoo}, and Muon~\cite{jordan2024muon}, with a unified framework covering several of the mentioned algorithms recently proposed in~\cite{gratton2026unified}. Adaptive gradient methods represent state-of-the-art algorithms to solve large machine learning problems where the cost of computing the objective function is sometimes prohibitive. Among these methods, AdaGrad has gain particular attention due to its flexibility and the simplicity arguments used in its convergence theory. Despite the popularity of adaptive gradient methods in standard first-order optimization, very few works apply adaptive step size selection strategies in the context of BCG, and the convergence rate of such algorithms is still unexplored~\cite{liu2023improved}. 

\paragraph*{Outline and Contributions}

We discuss relevant contributions on BCG in Section~\ref{sec:rel_work}; we describe the \name\ framework in Section~\ref{sec:algo} and  its convergence rate analysis in Section~\ref{sec:convergence}. In particular, we prove a general nonmonotone descent lemma  (Lemma~\ref{lemma:dec_bAdaGrad}), and we derive a unified convergence rate analysis  (Theorem~\ref{theorem:bAdag_GS}) that covers GS, UR, and C block selection strategies. In Section~\ref{sec:constrained}, we extend our algorithm and prove new convergence rates for box-constrained problems. Finally, we show exhaustive numerical results both on convex and nonconvex problems in Section~\ref{sec:exp}.

Our main contributions can be summarized as follows:
\begin{enumerate}
    \item We introduce \name: a novel adaptive BCG method inspired by AdaGrad. To our knowledge, \name\ is the first adaptive BCG algorithm with provable nonasymptotic convergence rates in the nonconvex setting covering GS, UR, and C block selection rules, while they are usually analyzed separately in the literature.
    \item We establish an ergodic sublinear convergence rate for the \name\ algorithm by proving that, under the GS and UR selection rules, the average gradient norm and its expectation respectively, decrease at a rate of $\calO\left(\frac{1}{\sqrt{k+1}}\right)$, where $k$ denotes the iteration index. Furthermore, for the cyclic variant of \name, we derive a sublinear convergence rate of $\calO\left(\frac{1}{\sqrt{T+1}}\right)$, where $T$ represents the complete cycle in which all blocks of variables are updated once.
    \item We extend the \name\ algorithm to box-constrained minimization and we establish similar rates to the unconstrained case. 
\end{enumerate} 
Finally, we illustrate our findings with exhaustive numerical results on several applications, including least square and logistic regression, and robust Nonnegative Matrix Factorization (NMF).

\paragraph*{Notation}   
We denote $v_{i,k}$ the $i$-th component of a vector $v_k\in \Re^n$. The gradient of $f$ is denoted as $g\eqdef\nabla f$; thus, at the $k$-th iteration, the gradient evaluated at the current iterate $x_k$ is $g_k = g(x_k)$ and its $i$-th component is $g_{i,k}$. We introduce a partition of the variable into disjoint blocks denoted 
$
\{1,\dots,N\} = \bigcup_{\ell=1}^d b_\ell,
$
where $b_\ell \subseteq \{1,\dots,N\}$, $b_\ell \cap b_s = \emptyset$ for $\ell \neq s$.
We denote  $g_{i,k}^{(\ell)}$ the block gradient containing only the components in $b_\ell$ and being zero elsewhere, that is, 
\begin{equation}
    g_{i,k}^{(\ell)}=\left\{\begin{array}{cc}
    g_{i,k} & i \in b_{\ell} \\
    0 & \text{otherwise.}
\end{array}\right.  
\label{gOxD-def}
\end{equation}
 Unless specified otherwise, $\|\cdot\|$ is the Euclidean norm on $\Re^n$. Since in this paper we deal with nonsynmptotic rates, given two sequences $\{\alpha_k\}$  and $\{\beta_k\}$ of non-negative reals, we write with a slight abuse of notation, $\alpha_k =\calO(\beta_k)$ if there exists a finite constant $C$ such that $\alpha_k \leq C \beta_k $ for every $k\geq 1$.

\section{Related works}
\label{sec:rel_work}
Despite that BCD approaches have long been known, they have recently regained interest in the study of large-scale problems. They can be classified into two categories: exact and inexact algorithms. The former is characterized by solving each block subproblem to global optimality and convergence results often rely on convexity or relaxed per-block convexity assumptions on the objective function~\cite{auslender1976optimisation,bertsekas1997nonlinear,grippo2000convergence}. While in the latter, the block subproblems are solved approximately; thus, the BCG method is an example of inexact scheme.

The convergence properties of BCG have been extensively studied in the convex setting, and the corresponding theory is by now well established, see~\cite{wright2015coordinate,shi2016primer} for review articles on the topic. Among all possible block selection strategies, randomized BCG has attracted the largest body of contributions due to the simplicity of the arguments employed in its convergence analysis. In particular, for randomized (accelerated) Coordinate Gradient (CG), Nesterov established in~\cite{nesterov2012efficiency} nonasymptotic sublinear convergence rates in the convex case and linear convergence rates in the strongly convex setting. Following the work of Nesterov, sharper rates were derived in~\cite{lee2013efficient,allen2016even,nesterov2017efficiency}, by introducing different sampling techniques. In the nonsmooth setting, a first sublinear rate was established for an $\ell_1$ regularized finite sum minimization problem in~\cite{shalev2009stochastic}. Later, the result of Nesterov was extended in~\cite{richtarik2014iteration} and further refined in~\cite{lu2015complexity} to Block Coordinate Proximal Gradient (BCPG) for the sum of a smooth and a nonsmooth, block-separable, convex function. Subsequent works have established convergence rates for randomized BCG algorithms in specialized convex settings, including parallelizable extensions~\cite{necoara2016parallel} and formulations involving nonseparable objective functions~\cite{aberdam2022accelerated}.

Despite the analysis of cyclic BCG being more involved than that of randomized BCG, Luo and Tseng proved in~\cite{luo1992convergence} an asymptotic linear convergence rate for cyclic CG for strictly convex and twice differentiable functions. Saha and Tewari obtained in~\cite{saha2010finite} the first nonasymptotic result with a sublinear rate in the convex case under an isotonicity assumption on the gradient. Later, Beck and Tetruashvili in~\cite{beck2013convergence} proved sublinear convergence of cyclic BCG for a general smooth and convex function. Their bounds were refined in~\cite{sun2015improved} for the sum of a quadratic and a block-separable, nonsmooth term, and for general convex functions, employing second order information. A similar composite setting was also considered in~\cite{li2018faster}. Other relevant contributions focus on the quadratic case,  see~\cite{lee2019random,sun2021worst,wright2020analyzing}. Moreover, cyclic accelerated BCPG was extended  to the class of Minty variational inequalities with monotone Lipschitz operators, encompassing convex composite optimization as well as convex-concave min-max problems. By introducing a generalized Lipschitz condition with respect to a Mahalanobis norm, the authors achieved the significant improvement of reducing the dependence of the convergence rate on the number of blocks by a square-root factor~\cite{song2023cyclic}. To the best of our knowledge, the first accelerated BCPG method with adaptive step sizes was recently proposed by Wei et al.~\cite{wei2026adaptive}. The authors considered the same class of Minty variational inequalities with monotone Lipschitz operators, where the step sizes are adaptively selected according to local curvature information. Their framework is also parallelizable since it utilizes operator information delayed by a full cycle. Similar to the method proposed in this paper, their algorithm does not require prior knowledge of the Lipschitz constant nor the use of a line-search procedure. Nevertheless, their approach differs substantially from ours: it is not an OFFO algorithms, in fact the local estimates of the Lipschitz constant rely on explicit function evaluations. Moreover, their framework does not cover general nonconvex functions.
 
Greedy selection rules have been shown to have better convergence rates than randomized rule when minimizing smooth strongly convex~\cite{nutini2015coordinate} as well as general convex functions~\cite{stich2017approximate}. On the other hand, the per-iteration computational cost of GS is usually higher than that of random strategies, since it requires the evaluation of the full gradient. Thus, random and cyclic rules are often preferred in practice unless greedy strategies can be cheaply evaluated at every iteration, see~\cite{stich2017approximate,lei2016coordinate,nutini2015coordinate}. Nevertheless, a sublinear rate of convergence for the BCG method with GS-like block selection rules has been derived for convex problems~\cite{tseng2009block,dhillon2011nearest}. Another relevant contribution covering both greedy and cyclic rules is the block successive upper-bound minimization (BSUM) framework, which provides a unified convergence rate analysis for both the BCG and the BCPG methods for convex objectives~\cite{razaviyayn2013unified,hong2017iteration}. Relevant contributions providing nonasymptotic convergence rates for BCG and related algorithms in the convex setting are summarized in Table~\ref{tab:rel_work}. Moreover, several works analyze the more general setting of composite problems of the form $f+h$, where usually $f$ is smooth and $h$ is nonsmooth; thus we include them in the table specifying, in the third column, the additional assumptions for the nonsmooth term $h$. Due to space limitations, the table does not cover specialized settings such as the strongly convex or constrained, which are possibly separately analyzed by some of the contribution. Similarly, since all the contributions also assume (block) Lipschitz continuity of the gradient or generalized variants we omit this information from Table~\ref{tab:rel_work}.
 
\begin{table}[h]
    \centering
    \resizebox{\textwidth}{!}{
   \begin{tabular}{lccccccc}
   \hline
      \multicolumn{6}{c}{ \textbf{Coordinate descent in convex setting with nonasyntotic rates guarantees}} \\  
   \hline
    & Algorithm & Block rule  & \makecell{Nonsmooth \\term}  & \makecell{Adaptive \\ step}  &  \makecell{Additional \\ assumptions} \\
     \hline
    \vspace{0.1cm}
    Nesterov~\cite{nesterov2012efficiency} & CG/ACG & R &  /& no &  /\\
    \hline
    \vspace{0.1cm}
    Tee, Sidford~\cite{lee2013efficient} & ACG & R &  /& no &  \makecell{strongly convex $f$} \\
    \hline
    \vspace{0.1cm}
    Allen-Zhu, al.~\cite{allen2016even} & ACG & R  &  /& no &  / \\
    \hline
    \vspace{0.1cm}
    Nesterov, Stich~\cite{nesterov2017efficiency} & ACG & R  &  /& no &  / \\
    \hline
    \vspace{0.1cm}
    Shalev-Shwartz, Tewari~\cite{shalev2009stochastic} & PCG & R  &  $\lVert \ \cdot  \ \rVert_1$& no & $f$ in a finite-sum form  \\
    \hline
    \vspace{0.1cm}
    Richtárik, Tak\'a\v{c}~\cite{richtarik2014iteration} & BCPG & R  &  separable & no & /  \\
    \hline
    \vspace{0.1cm}
    Lu, Xiao~\cite{lu2015complexity} & BCPG/ABCG & R  &  separable& no & /  \\
     \hline
     \vspace{0.1cm}
    Saha, Tewari~\cite{saha2010finite} & PCG  & C   &  $\lVert \ \cdot  \ \rVert_1$ & no &  isotonicity of $\nabla f$ \\
    \hline
    \vspace{0.1cm}
    Beck, Tetruashvili~\cite{beck2013convergence} & BCG/ABCG & C &  / & no & /  \\
     \hline
    \vspace{0.1cm}
    Sun, Hong~\cite{sun2015improved} & CG  & C &  / & no & twice differentiable $f$ \\
    \hline
    \vspace{0.1cm}
    Sun, Ye~\cite{sun2021worst} & CG & C &  / & no  & quadratic $f$\\
     \hline
    \vspace{0.1cm}
    Lee, Wright~\cite{lee2019random} & CG & C &  / & no &  quadratic $f$\\
    \hline
    \vspace{0.1cm}
    Li et al.~\cite{li2018faster} & BCPG & C &  \makecell{strongly convex} & no & twice differentiable $f$  \\
    \hline
    \vspace{0.1cm}
    Song and Diakonikolas~\cite{song2023cyclic}\footnotemark[1] & ABCPG & C &  separable & no  & /   \\
    \hline
    \vspace{0.1cm}
    Wei et al.~\cite{wei2026adaptive}\footnotemark[1] & ABCPG & C &  separable & yes  & /  \\
    \hline
    \vspace{0.1cm}
     Razaviyayn et al.~\cite{hong2017iteration} & BCPG\footnote{The paper presents a unified framework that covers several BCD algorithm, see the reference for more details.} & C/Gr &  separable & no & / \\
    \hline
     \vspace{0.1cm}
     Tseng, Yun~\cite{tseng2009block} & BCG & Gr &  separable & no & /  \\
    \hline
     \vspace{0.1cm}
      Dhillon et al.~\cite{dhillon2011nearest} & PCG & Gr &  separable & no  & / \\
    \hline
\end{tabular}
}
    \caption{Summary of contributions establishing nonasymptotic convergence rates for BCG and its variant in the convex setting. Notation in the first column: Coordinate Gradient (CG), Accelerated CG (ACG), Proximal CG (PCG), Block CG (BCG), Accelerated BCG (ABCG), Block Coordinate Proximal Gradient (BCPG); in the second column: Random (R), Cyclic (C), Greedy (Gr). The third columns shows if composite (smooth+nonsmooth) optimization problems of the form $f+h$ are considered and, if so, the type of nonsmooth term $h$. The last column contains additional assumptions considered in the paper.  }
    \label{tab:rel_work}
\end{table}

\footnotetext[1]{Their framework analyzes general monotone variational inequalities; thus, covering some structured nonconvex problems such as convex-concave min-max optimization.}
The literature on the non-asymptotic convergence of BCG in the nonconvex setting is much more limited, while most existing results focus on asymptotic convergence, see~\cite{attouch2013convergence,lu2013randomized,xu2013block,razaviyayn2013unified,bonettini2016cyclic,xu2017globally,latafat2022block}. Among the relevant contributions investigating nonasymptotic convergence rates in the nonconvex setting, Csiba and Richtárik in~\cite{csiba2017global} proposed a unified analysis for an arbitrary block-selection rule, under the Polyak-Łojasiewicz (PL) or Weakly PL (WPL) properties, proving a linear convergence and a sublinear rate, respectively. For randomized BCG, Patrascu and Necoara in~\cite{patrascu2015efficient} established a sublinear convergence rate for BCPG in a more general setting, specifically when minimizing the sum of a nonconvex, smooth function and a convex, separable, and possibly nondifferentiable function. Sublinear convergence rates with high probability were also derived in~\cite{necoara2025efficiency} for (stochastic) random BCG for smooth but nonseparable functions.
 To our knowledge, for cyclic (proximal) BCG, the only nonasymptotic convergence guarantee in general nonconvex setting is given by Cai et al. in~\cite{cai2023cyclic}. In particular, they analyze a composite minimization problem where the smooth part consists of a finite sum of functions and the possibly nonsmooth part is block-separable. Under a non restrictive Lipschitz condition on the gradient with respect to a Mahalanobis norm (implied by  the global Lipschitz continuity of the gradient), they could prove sublinear convergence of cyclic BCPG. 
 For greedy BCG and BCPG, Nutini et al. in~\cite{Nutini_faster} considered different GS-like block selection rules proving sublinear and linear convergence in the general nonconvex setting and under the PL condition on the objective function, respectively. 
 
In all the contributions listed above, the step size in the BCG update is chosen either by assuming some knowledge of the (block) Lipschitz constant, which might not be the case in practical applications, or by using a line-search strategy. Recently, Liu et al. showed in~\cite{liu2023improved} the potential of combining BCG and adaptive step sizes. They proposed a block coordinate version of stochastic Adam but the convergence theory of their method has not been investigated. Moreover, a unified convergence theory for adaptive first-order methods has been very recently proposed in~\cite{gratton2026stochastic} in the related context of stochastic asynchronous gradient methods with delayed operators. We summarized the contributions in the nonconvex setting in Table~\ref{tab:rel_work_nonconvex}. For a more detailed discussion and comparison among convergence rates, see Section~\ref{sub:comp_works}.

\begin{table}[h]
    \centering
    \resizebox{\textwidth}{!}{
   \begin{tabular}{lccccccc}
   \hline
      \multicolumn{7}{c}{ \textbf{BCG global nonasymptotic rates for nonconvex functions}} \\  
   \hline
    & Block rule & \makecell{Nonsmooth \\extension}  &\makecell{Adaptive \\ step}  & \makecell{Additional assumptions}  & Rate for smooth objective &\\
    \hline
    \vspace{0.1cm}
    Csiba, Richtárik~\cite{csiba2017global} &  R/Gr  & \cmark & no & $f$ is PL & $\E [f(x^k)-f^*] = \mathcal{O}(\Theta^k)$ \\
    \hline
     \vspace{0.1cm}
    Csiba, Richtárik~\cite{csiba2017global} &  R/Gr  & \cmark & no & $f$ is WPL & $\E [f(x^k)-f^*] = \mathcal{O}(\frac{1}{k})$ \\
    \hline
     \vspace{0.1cm}
     Patrascu, Necoara~\cite{patrascu2015efficient}
      & R & \cmark & no & / &  $\min_{j=1,\dots,k}\E [\lVert \nabla f (x^j)\rVert^2] = \mathcal{O}\left(\frac{1}{k}\right)$\\
     \hline
     \vspace{0.1cm}
     Necoara, Chorobura~\cite{necoara2025efficiency} &  R & \xmark & no & / & $\min_{j=1,\dots,k}\lVert \nabla f (x^j)\rVert^2 = \mathcal{O}\left(\frac{1}{k}\right)$ w.h.p \\
    \hline
     \vspace{0.1cm}
    Nutini et al.~\cite{Nutini_faster}
    & Gr  & \cmark & no  & $f$ is PL  & $f(x^k)-f^* = \mathcal{O}(\Theta^k)$  \\
     \hline
     \vspace{0.1cm}
    Nutini et al.~\cite{Nutini_faster} &  Gr  & \cmark & no  & / & $\min_{j=1,\dots,k}\lVert \nabla f (x^j)\rVert^2 = \mathcal{O}\left(\frac{1}{k}\right)$ \\
    \hline
     \vspace{0.1cm}
     Cai et al.~\cite{cai2023cyclic}  & C & \cmark & no & $f$ in finite sum form & $\min_{j=1,\dots,k}\lVert \nabla f (x^j)\rVert^2 = \mathcal{O}\left(\frac{1}{k}\right)$ \\
     \hline
     \vspace{0.1cm}
     This paper  & R & \xmark & yes & / & $\E \left[ \average_{j=1,\dots,k}\lVert \nabla f (x^j)\rVert \right] = \mathcal{O}\left(\frac{1}{\sqrt{k}}\right)$ \\
     \hline
     \vspace{0.1cm}
     This paper  & C/Gr & \xmark & yes & / & $\average_{j=1,\dots,k}\lVert \nabla f (x^j)\rVert = \mathcal{O}\left(\frac{1}{\sqrt{k}}\right)$ \\
     \hline
\end{tabular}
}
   \caption{Summary of contributions establishing nonasymptotic convergence rates for the BCG for smooth and nonconvex functions. Notation in the first column: Random (R), Cyclic (C), Greedy (Gr). The fourth column contains additional assumptions made on the objective function with notation: Polyak–Łojasiewicz (PL), Weakly PL (WPL). All the contributions also assume (block) Lipschitz continuity of the gradient or other closely related assumptions which are not shown in this table. The last column contains the type of nonasymptotic convergence result, where $f^*$ denotes the optimal function value and with high probability is abbreviated to w.h.p. The norm used in the rate is not specified and it may vary depending on the contribution; thus, we refer to the original references for more details.}
    \label{tab:rel_work_nonconvex}
\end{table}

Other relevant contributions in the BCG literature succeeded in relaxing the hypothesis of block-Lipschitz continuity of the gradient by introducing weaker hypothesis, often based on Bregman distance~\cite{hanzely2021fastest,gao2019leveraging,9414191,hanzely2021accelerated,hien2025block}; or considered the more restrictive setting of alternate minimization, where the number of blocks is fixed at two~\cite{attouch2010proximal,bolte2014proximal,beck2015convergence,diakonikolas2018alternating}. Analyzing these rather different settings is out of the scope of this work, which focuses on the smooth, block-Lipschitz continuous, and possibly nonconvex case. Therefore, these contributions do not appear in Table~\ref{tab:rel_work} and~\ref{tab:rel_work_nonconvex}.

\section{The \name\ framework }
\label{sec:algo}
We proceed now by introducing the \name\ algorithm (Algorithm~\ref{alg:bAdaGrad}).
Let $\{1,\dots,N\} = \bigcup_{\ell=1}^d b_\ell$ be a partition of the variables into blocks which we assume remains fixed within the iterations.
Let also $g_k^{(\ell_k)}$ the block gradient defined in~\eqref{gOxD-def}. We consider three block selection rules in our analysis: 
\begin{itemize}
    \item \textbf{GS}: pick the block $b_{\ell_k}$ with the largest gradient norm, corresponding to
    \begin{equation*}
     \lVert g_k^{(\ell_k)} \rVert=\argmax_{\ell=1,\dots,d}  \lVert g_k^{(\ell)} \rVert.
    \end{equation*}
    \item \textbf{C}: pick the block $b_{\ell_k}$ in cyclic order; thus, every $d$ iterations, all the blocks are updated once. 
    \item \textbf{UR}: pick the block  $b_{\ell_k}$ uniformly at random, with a probability of $1/d$.
\end{itemize}
One can notice that the GS rule requires the evaluation of the full gradient at each iteration; thus, being in general more expensive than the UR and the cyclic rule. We are now ready to introduce the \name\ algorithm.
\algo{alg:bAdaGrad}{\tal{\name}}
{
\begin{description}
\item[Step 0: Initialization. ]
  A starting point $x_0$, a constant $\varsigma\in(0,1)$, and the initial weight vector $w_{i,-1}=\sqrt{\varsigma}$ for $i=1,\dots,N$, a fixed partition $\{1,\dots,N\} = \bigcup_{\ell=1}^d b_\ell$. Set $k=0$. 

\item[Step 1: Choose the block.] Select the block $b_{\ell_k} \subseteq \{1,\dots,N\}$ to be updated, according to the GS, UR or C rule.
 \item[Step 2: Gradient] If not already computed in Step 1, compute the gradient $g_{i,k}$ for $i \in b_{\ell_k}$.
  \item[Step 3: Optimization weights.] Compute 
  \beqn{weig_R_det}
  \begin{aligned}
      & w_{i,k} = \sqrt{(w_{i,k-1})^2+g_{i,k}^2}
      \quad \quad (i \in b_{\ell_k}),\\
      & w_{i,k}=w_{i,k-1} \quad \quad (i \in \{1,\dots,N\} \backslash b_{\ell_k}).
  \end{aligned}
   \eeqn
   
\item[Step 4: Compute the step.] Compute
  \beqn{step_R_det}
  \begin{aligned}
       &s_{i,k}=-\frac{g_{i,k}}{w_{i,k}} \quad \quad (i \in b_{\ell_k}).\\
       & s_{i,k}=0 \quad \quad (i \in \{1,\dots,N\} \backslash b_{\ell_k}).
  \end{aligned}
  \eeqn


\item[Step 5: New iterate.] Define
   \begin{equation}
       \qquad x_{k+1} = x_k + s_k, 
       \label{eq:iter_update}
   \end{equation}  
    increment $k$ by one and return to Step~1.
\end{description}
}

The definition of the adaptive step size used in \req{weig_R_det} at step 3 of Algorithm~\ref{alg:bAdaGrad} follows the same principle of the AdaGrad algorithm, with the crucial difference that only the weights in the selected block are updated. This means that the algorithm does not accumulate the entire history of the gradient but only the history of the gradient in the blocks selected by the BCD scheme. Indeed, using \req{weig_R_det} and the notation in~\eqref{gOxD-def}, one verifies that 

\begin{equation}
    w_{i,k} = \sqrt{\varsigma + \sum_{j=0}^k \left(g_{i,j}^{(\ell_j)} \right)^2}, \quad \text{and } \quad s_{i,k}=-\frac{g_{i,k}^{(\ell_k)}}{w_{i,k}}.
    \label{eq:explicit_weights}
\end{equation}
It is clear from~\eqref{gOxD-def} and~\eqref{eq:explicit_weights} that the step $s_{i,k}$ is non-zero only in the selected block $b_{\ell_k}$; thus, the update of the iterate $x_k$ in~\eqref{eq:iter_update} only affects the component in the block, and it is of the same form as the one in~\eqref{eq:block_grad_step}. 

\section{Convergence}
\label{sec:convergence}
We now present the nonasymptotic convergence analysis of \name. We develop a unified framework that covers GS, UR, C. We first characterize the descent property of the algorithm by bounding the difference in the function values of two consecutive iterates in Lemma~\ref{lemma:dec_bAdaGrad}. Then, we provide a unified strategy to upper bound the average norm of the block gradient, independently from the block selection rule considered, and finally we derive specialized bounds on the average norm of the full gradient for each block selection rule (Theorem~\ref{theorem:bAdag_GS}).

Let us introduce standard assumptions in the nonasymptotic convergence analysis of first order BCG algorithms.
\begin{description}
\item[AS.1:] the objective function $f$ is continuously differentiable;
\item[AS.2:] there exists a constant $\flow$ such that, for all $x$, $f(x)\ge \flow$;
\item[AS.3:] given a partition of the components of the iterate in blocks $\{1,\dots,N\} = \bigcup_{\ell=1}^d b_\ell$, the gradient $g$ is block Lipschitz continuous with block Lipschitz constant $L_\ell \geq 0$, that is
  \begin{equation}
       \|g^{(\ell)}(x)-g^{(\ell)}(x+s^{(\ell)})\| \le L_\ell \|s^{(\ell)}\| \quad \text{for every } \ell=1,\dots,d,
       \label{eq:block_lip}
  \end{equation}
  
   for all $x\in \Re^N$; where $s^{(\ell)} \in \Re^N$ denotes a vector whose support is contained in $b_{\ell}$. It follows from AS.3 that
\begin{equation}
    f(x+s^{(\ell)}) \leq  f(x)+  g(x)^T s^{(\ell)} +  \frac{L_{\ell}}{2} ||s^{(\ell)}||^2. 
    \label{eq:taylor_app}
\end{equation}

We also define the Lipschitz continuity of the full gradient, which is only needed to prove convergence for cyclic \name.
\item[AS.4:] the gradient $g$ is Lipschitz continuous with global Lipschitz constant $L\geq 0$:
   \[
   \|g(x)-g(y)\| \le L \|x-y\| \quad \text{ for all } x,y\in \Re^N.
   \]   
\end{description}
Notice that AS.3 is implied by AS.4 by chosing $L_\ell=L$, for every $\ell=1,\dots,d$. 

Moreover, the UR rule introduces randomness in the algorithm; thus, the sequence generated by \name\ is a random process. As a consequence, the convergence analysis for the \name\ with the UR rule is in expectation. Specifically, at iteration $k$, the expectation conditioned to knowing the history of the past iterate $x_0, \dots, x_{k-1}$ will be denoted by the symbol $\E_k[\cdot]$. Moreover, when the index $\ell_k$ of the block to optimize at each iteration is chosen uniformly at random,  it holds
 \begin{equation}
     \E_k[\lVert g_k^{(\ell_k)} \rVert]=  \frac{1}{d} \sum_{\ell=1}^d \lVert g_k^{(\ell)} \rVert.
     \label{eq:unbias}
 \end{equation}

Furthermore, the block gradient $g_k^{(\ell_k)}$ satisfies $\E_k[ g_k^{(\ell_k)}]= 1/d \sum_{\ell=1}^d  g_k^{(\ell)} = 1/d \ g_k$, which means that the block gradient is an unbiased estimator of the true gradient, except for a constant factor $d$. Thus, the analysis of  \name\ with the UR rule is partially inspired by that of the stochastic AdaGrad algorithm in~\cite{gratton2025complexity}. Nevertheless, there are some crucial differences:
\begin{itemize}
    \item[a)] we do not consider any rescaling in our algorithm; therefore, the theory in~\cite{gratton2025complexity} does not apply when the stochastic estimator of the gradient is $g_k^{(\ell_k)}$; 
    \item[b)] Contrary to~\cite{gratton2025complexity}, we do not assume boundedness of the gradient in our analysis. 
\end{itemize}

We are now ready to state the descent lemma for the \name\ algorithm. The proof of the lemma is provided only for the UR rule since that for GS and C follows from the same idea by removing expected values. In fact, both the rules do not introduce any randomness in the algorithm; thus, all the quantities in expectation are deterministic.  

\begin{lemma}
    \label{lemma:dec_bAdaGrad}
    Suppose that AS.1 - AS.3  hold. Then, according to the block selection rules, we have that, for all $j\ge0$,
\begin{equation}
  \text{\textbf{(GS/C) :}} \quad    f(x_{j+1})
\le f(x_j) -\sum_{i=1}^N \frac{ (g_{i,j}^{(\ell_j)})^2}{ w_{i,j}}
     + \frac{L_{\ell_j}}{2} \sum_{i=1}^N \frac{ (g_{i,j}^{(\ell_j)})^2}{w_{i,j}^2}, 
      \label{gen-decr-det-gs-c} 
\end{equation}
\begin{equation}
    \text{\textbf{(UR) :}} \quad  \E_j\left[f(x_{j+1})\right]
\leq f(x_j) -\E_j\left[\sum_{i=1}^N \frac{ (g_{i,j}^{(\ell_j)})^2}{ w_{i,j}} \right]
      +\frac{L_{\ell_j}}{2}  \E_j\left[ \sum_{i=1}^N \frac{ (g_{i,j}^{(\ell_j)})^2}{w_{i,j}^2} \right]. 
    \label{gen-decr-det}
\end{equation}

    
\end{lemma}

\proof{
Using Assumptions AS.1 and AS.3, the definition of the step in \req{step_R_det}, and the inequality in \req{eq:taylor_app}, we derive that
\beqn{f_dec}
\begin{aligned}
\E_j\left[f(x_{j+1}) \right]&\leq  f(x_j)+ \E_j\left[ g_j^T s_j \right]+  \frac{L_{\ell_j}}{2} \E_j\left[ ||s_j||^2 \right] \\
&= f(x_j)+ \E_j\left[ \sum_{i\in b_{\ell_j}} g_{i,j}s_{i,j} \right]+  \frac{L_{\ell_j}}{2} \E_j\left[ \sum_{i\in b_{\ell_j}}s_{i,j}^2 \right] \\
&= f(x_j) - \E_j\left[ \sum_{i \in b_{\ell_j}} \frac{ g_{i,j}^2}{ w_{i,j}} \right]
      +\frac{L_{\ell_j}}{2} \E_j\left[ \sum_{i \in b_{\ell_j}} \frac{ g_{i,j}^2}{w_{i,j}^2} \right]\\
&=f(x_j) - \E_j\left[ \sum_{i = 1}^N \frac{ (g_{i,j}^{(\ell_j)})^2}{ w_{i,j}} \right]
     + \frac{L_{\ell_j}}{2} \E_j\left[ \sum_{i=1}^N \frac{ (g_{i,j}^{(\ell_j)})^2}{ w_{i,j}^2}      \right],
\end{aligned}
\eeqn
which concludes the lemma.
}

We introduce two technical lemmas that are crucial to prove our main convergence result.  

\begin{lemma}[\cite{bellavia2025fast}]
    \label{log_bound}
    Suppose that, for $u>0$ and $\gamma_1,\gamma_2>0$,
\begin{equation*}
    \gamma_1u \leq \gamma_2 \log(u) + \gamma_3.
\end{equation*}
Then,
\begin{equation*}
    u\leq \frac{2 \gamma_3}{\gamma_1} + \frac{2 \gamma_2}{\gamma_1} \left( \log \left( \frac{2 \gamma_2}{\gamma_1} \right) -1\right).
\end{equation*}
\end{lemma}

The next lemma is from~\cite{techlemma1,techlemma2,WuWardBott18} and we use it to bound the contribution of the quadratically decaying term in the descent lemma of the \name\ algorithm.

\begin{lemma}
    \label{gen:series}
    Let $\{c_k\}_{k\ge 0}$ be a non-negative sequence and  $\xi>0$.
Then
\begin{equation*}
   \sum_{j=0}^k  \frac{c_j}{(\xi+ \sum_{q=0}^j c_{q})}
\le  \log\left(1+\frac{1}{\xi} \sum_{j=0}^k c_j \right). 
\end{equation*}
\end{lemma}

The proof of the lemma can be found in~\cite[Lemma 3.1]{gratton2024complexity}. 

The following lemma aims at providing a lower bound for the first summation in~\eqref{gen-decr-det-gs-c} and~\eqref{gen-decr-det} which linearly decreases with the weights $w_{i,j}$.  

\begin{lemma}
    \label{lem:first_sum}
    Assume \name\ is applied to problem~\eqref{eq:problem}, then
$$ \sqrt{\varsigma}  \sqrt{ 1 + \frac{1}{\varsigma}  \sum_{j=0}^k \|g_j^{(\ell_j)}\|^2 } - \sqrt{\varsigma} \leq \sum_{j=0}^k\sum_{i =1}^N \frac{ (g_{i,j}^{(\ell_j)})^2}{ w_{i,j}}. $$
\end{lemma}

\proof{
Since the sequence of $w_{i,j}$ in \req{eq:explicit_weights} is increasing in $j$, $w_{i,j} \leq w_{i,k}$ for every $j=0,\dots,k$. Then 
\begin{equation}
    w_{i,j} \leq w_{i,k} =  \sqrt{ \varsigma + \sum_{q=0}^k (g_{i,q}^{(\ell_q)})^2} \leq \sqrt{ \varsigma + \sum_{q=0}^k \lVert g_{q}^{(\ell_q)} \rVert^2},
\end{equation}
from which we obtain
\begin{equation} 
     \sum_{j=0}^k \frac{ \|g_j^{(\ell_j)}\|^2}{ \sqrt{ \varsigma + \sum_{q=0}^k \lVert g_{q}^{(\ell_q)} \rVert^2}}   \leq \sum_{j=0}^k\sum_{i =1}^N \frac{ (g_{i,j}^{(\ell_j)})^2}{ w_{i,j}},
     \label{eq:first_step_lem}
\end{equation}

Furthermore, we have
\begin{equation*}
\begin{aligned}
   \sum_{j=0}^k \frac{ \|g_j^{(\ell_j)}\|^2}{ \sqrt{ \varsigma + \sum_{q=0}^k \lVert g_{q}^{(\ell_q)} \rVert^2}} = &  \frac{ \sum_{j=0}^k \|g_j^{(\ell_j)}\|^2}{ \sqrt{ \varsigma + \sum_{q=0}^k \lVert g_{q}^{(\ell_q)} \rVert^2}} + \frac{ \varsigma}{ \sqrt{ \varsigma + \sum_{q=0}^k \lVert g_{q}^{(\ell_q)} \rVert^2}}  - \frac{ \varsigma}{ \sqrt{ \varsigma + \sum_{q=0}^k \lVert g_{q}^{(\ell_q)} \rVert^2}} \\
   = & \sqrt{ \varsigma + \sum_{j=0}^k \lVert g_{j}^{(\ell_j)} \rVert^2} - \frac{ \varsigma}{ \sqrt{ \varsigma + \sum_{q=0}^k \lVert g_{q}^{(\ell_q)} \rVert^2}} \\
    \geq & \sqrt{ \varsigma + \sum_{j=0}^k \lVert g_{j}^{(\ell_j)} \rVert^2} - \sqrt{\varsigma},\\
    = & \sqrt{\varsigma}\cdot \sqrt{ 1 + \frac{1}{\varsigma}  \sum_{j=0}^k \|g_j^{(\ell_j)}\|^2 } - \sqrt{\varsigma},
\end{aligned}
\end{equation*}
which, combined with~\eqref{eq:first_step_lem}, concludes the proof.
}
We are now ready to state our main result, characterizing the nonsymptotic rate of convergence of \name. 

\begin{theorem}
    \label{theorem:bAdag_GS}
    Suppose that AS.1--AS.3 hold and that
\name\ is applied to problem \req{eq:problem}.
Let $\varsigma>0$, $d$ the number of blocks, and
$\Gamma_0 \eqdef f(x_0)-\flow.$ Then,
\begin{itemize}
    \item \textbf{GS:}
    \begin{equation}
        \average_{j\in\iiz{k}}\|g_j\| \le \sqrt{d} \frac{\Theta }{\sqrt{k+1}}.
        \label{gradbound_det}
    \end{equation}
    \item \textbf{C:} if, in addition, AS.4 holds, then
    \begin{equation}
        \average_{t \in \{0,\dots,T\}} \lVert g_{td} \rVert \leq  \sqrt{2  \left(  1 + d\frac{L^2}{\varsigma^2} \right)} \frac{\Theta }{\sqrt{T+1}}.
        \label{eq:rate_cyclic}
    \end{equation}
    \item \textbf{UR:}
    \begin{equation}
        \E \left[ \average_{j\in\iiz{k}}\|g_j\| \right] \le d \frac{  \Theta}{\sqrt{k+1}},
        \label{gradbound}
    \end{equation}
\end{itemize}

with \begin{equation}
    \label{k6-def}
    \Theta \eqdef 2\Gamma_0 + 2\sqrt{\varsigma} + 2N\Bar{L}\left[ \log\left( \frac{2N\Bar{L}}{\sqrt{\varsigma}} \right) -1\right],
\end{equation}

where $\Bar{L}=\max_{\ell=1,\dots,d}L_\ell$.
\end{theorem}

\proof{Since the descent lemma for the GS/C (deterministic) and UR rule (stochastic) differ, we analyze separately these two cases. However, the proof follows the same strategy, with some more technical steps needed in the UR selection case since expected values are involved.

\textbf{GS/C:} Starting from the descent lemma in~\eqref{gen-decr-det-gs-c} and summing for $j=0,\dots,k$,
\begin{equation*}
    f(x_0)-f(x_{k+1}) \geq 
\sum_{j=0}^k \sum_{i=1}^N \frac{(g_{i,j}^{(\ell_j)})^2}{ w_{i,j}}
   - \frac{\Bar{L}}{2} \sum_{j=0}^k \sum_{i=1}^N \frac{ (g_{i,j}^{(\ell_j)})^2}{w_{i,j}^2},
\end{equation*}
where $\Bar{L}=\max_{\ell=1,\dots,d}L_\ell$.
Rearranging the terms and adding the definition of $\Gamma_0$, 
\beqn{dec_grad_det}
\sum_{j=0}^k\sum_{i =1}^N \frac{ (g_{i,j}^{(\ell_j)})^2}{ w_{i,j}}
\leq \Gamma_0 + \frac{\Bar{L}}{2} \sum_{j=0}^k \sum_{i=1}^N \frac{ (g_{i,j}^{(\ell_j)})^2}{w_{i,j}^2}.
\eeqn
We first focus on computing an upper bound of the right hand side. In particular, we use the reformulation of $w_{i,j}$ in~\eqref{eq:explicit_weights} and Lemma \ref{gen:series} with $c_{j}=(g_{i,j}^{(\ell_j)})^2$ and $\xi = \zeta$, yielding
\beqn{use_lemma_det}
\begin{aligned}
\sum_{i =1}^N \sum_{j=0}^k  \frac{ (g_{i,j}^{(\ell_j)})^2}{w_{i,j}^2}
&= \sum_{i =1}^N \sum_{j=0}^k  \frac{ (g_{i,j}^{(\ell_j)})^2}{\varsigma+\sum_{q=0}^j (g_{i,q}^{(\ell_q)})^2} 
\leq \sum_{i =1}^N\log   \left( 1 + \frac{1}{\varsigma}\sum_{j=0}^k (g_{i,j}^{(\ell_j)})^2 \right)\\
& \leq \sum_{i =1}^N \log\left( 1 + \frac{1}{\varsigma}  \sum_{j=0}^k \|g_j^{(\ell_j)}\|^2 \right) =  2 N \log\left( \sqrt{1 + \frac{1}{\varsigma}  \sum_{j=0}^k \|g_j^{(\ell_j)}\|^2 } \right).
\end{aligned}
\eeqn

Combining (\ref{dec_grad_det}) and (\ref{use_lemma_det}), gives
\beqn{final_grad_case1_det}
\sum_{j=0}^k\sum_{i =1}^N \frac{ (g_{i,j}^{(\ell_j)})^2}{ w_{i,j}} \leq \Gamma_0 + N\Bar{L}\log\left( \sqrt{1 + \frac{1}{\varsigma}  \sum_{j=0}^k \|g_j^{(\ell_j)}\|^2 } \right).
\eeqn
We now focus on the left hand side introducing the lower bound provided by Lemma~\ref{lem:first_sum}, which gives: 
\begin{equation*}
     \sqrt{\varsigma}  \sqrt{ 1 + \frac{1}{\varsigma}  \sum_{j=0}^k \|g_j^{(\ell_j)}\|^2 } \leq  \sqrt{\varsigma} + \Gamma_0 + N\Bar{L}\log\left( \sqrt{1 + \frac{1}{\varsigma}  \sum_{j=0}^k \|g_j^{(\ell_j)}\|^2 } \right) .
\end{equation*}
We can now apply Lemma~\ref{log_bound} with 
$$ u=   \sqrt{ 1 + \frac{1}{\varsigma}  \sum_{j=0}^k \|g_j^{(\ell_j)}\|^2 }, \quad \gamma_1= \sqrt{\varsigma}, \quad \gamma_2=N \Bar{L}, \quad \gamma_3= \sqrt{\varsigma} +  \Gamma_0, $$
obtaining 
\begin{equation}
     \sqrt{ 1 + \frac{1}{\varsigma}  \sum_{j=0}^k \|g_j^{(\ell_j)}\|^2 }  \leq \frac{2(\sqrt{\varsigma}+\Gamma_0)}{\sqrt{\varsigma}}+\frac{2N\Bar{L}}{\sqrt{\varsigma}}\left[ \log\left( \frac{2N\Bar{L}}{\sqrt{\varsigma}} \right) -1\right].
\end{equation}
Simple calculations yields
\begin{equation}
     \sqrt{ \sum_{j=0}^k \lVert g_{j}^{(\ell_j)} \rVert^2}  \leq  \sqrt{\varsigma}   \sqrt{ 1 + \frac{1}{\varsigma}  \sum_{j=0}^k \|g_j^{(\ell_j)}\|^2 }   \leq 2(\sqrt{\varsigma}+\Gamma_0) + 2N\Bar{L}\left[ \log\left( \frac{2N\Bar{L}}{\sqrt{\varsigma}} \right) -1\right].
     \label{eq:bound_glj}
\end{equation}
Let $\Theta=2\sqrt{\varsigma}+2\Gamma_0 + 2N\Bar{L}\left[ \log\left( \frac{2N\Bar{L}}{\sqrt{\varsigma}} \right) -1\right]$, dividing by $\sqrt{k+1}$, and using the inequality
\begin{equation}
    \frac{1}{k+1}\sum_{j=0}^k \lVert g_j^{(\ell_j)} \rVert \leq  \frac{1}{\sqrt{k+1}} \sqrt{\sum_{j=0}^k \lVert g_j^{(\ell_j)} \rVert^2},
    \label{eq:square_to_norm}
\end{equation}
we have
\begin{equation}
     \average_{j \in \{0,\dots,k\}} \lVert g_j^{(\ell_j)} \rVert \leq   \sqrt{\average_{j \in \{0,\dots,k\}} \lVert g_j^{(\ell_j)} \rVert^2} \leq \frac{ \Theta}{\sqrt{k+1}}.
    \label{eq:g_gR}
\end{equation}
The next step is to derive a bound for the full gradient norm rather than the norm of the block gradient. This bound depends on the block selection criteria. Thus, we now consider two different cases corresponding to the GS rule or the Cyclic rule.

\textbf{Case 1: GS.} At every iteration, the block $b_{\ell_j}$ corresponding to the largest gradient norm is chosen; thus, it trivially follows that for every $j$ 
\begin{equation}
     \lVert g_j \rVert = \sqrt{ \sum_{\ell=1}^d \lVert g_j^{(\ell)} \rVert^2} \leq \sqrt{ d  \lVert g_j^{(\ell_j)} \rVert^2}= \sqrt{d} \lVert g_j^{(\ell_j)} \rVert.
     \label{eq:GS_bound_grad}
\end{equation}
Thus, from~\eqref{eq:g_gR},
\begin{equation}
\average_{j\in\iiz{k}} \|g_j\| \leq \sqrt{d} \ \average_{j\in\iiz{k}} \lVert g_j^{(\ell_j)}\rVert\leq \frac{\sqrt{d} \Theta}{k+1},
\label{avg_case1_1_det}
\end{equation}
hence obtaining~\eqref{gradbound_det}.

\textbf{Case 2: C.} Let us first observe that in the cyclic rule, every $d$ iterations, all the blocks are updated the same number of times. In order to simplify the passages that follow, we reparametrize the iteration counter by grouping iterations into groups of $d$ elements. Equivalently, we write $k=Td+\ell$, where $\ell=0,\dots,d-1$. Consequently, all the components in $x_{td}$, for  $t=0,1,\dots, T$, have been updated the same number of times. Following the same reasoning presented in~\cite[Lemma 3.3]{beck2013convergence}, it holds the following
\begin{equation}
    x_{td+\ell} =x_{td}- \sum_{s=1}^{\ell-1} \frac{g^{(s)}_{td+s}}{w_{td+s}}.
    \label{eq:cum_xtd}
\end{equation}
Using~\eqref{eq:cum_xtd} and the Lipschitz continuity of $f$ in AS.4, the fact that $\varsigma \leq w_{td+\ell}$ for every $t$ and $\ell$, we have 
\begin{equation}
\begin{aligned}
     \lVert g_{td}^{(\ell)} \rVert^2 & \leq  \left( \lVert g_{td}^{(\ell)} -  g_{td+\ell}^{(\ell)}\rVert + \lVert g_{td+\ell}^{(\ell)} \rVert \right)^2 \leq 2  \lVert g_{td}^{(\ell)} -  g_{td+\ell}^{(\ell)} \rVert^2 + 2\lVert g_{td+\ell}^{(\ell)} \rVert^2\\
     &\leq  2  \lVert g_{td} -  g_{td+\ell}\rVert^2 + 2\lVert g_{td+\ell}^{(\ell)} \rVert^2 \leq 2L^2 \lVert x_{td} -  x_{td+\ell}\rVert^2 + 2\lVert g_{td+\ell}^{(\ell)} \rVert^2\\
     &\leq 2 L^2 \left\lVert \sum_{s=1}^{\ell-1} \frac{g^{(s)}_{td+s}}{w_{td+s}} \right\rVert^2 + 2\lVert g_{td+\ell}^{(\ell)} \rVert^2 \leq 2L^2 \sum_{s=1}^{\ell-1} \left\lVert \frac{g^{(s)}_{td+s}}{w_{td+s}} \right\rVert^2 + 2\lVert g_{td+\ell}^{(\ell)} \rVert^2\\
     & \leq \frac{2 L^2}{\varsigma^2} \sum_{s=1}^{\ell-1} \lVert g^{(s)}_{td+s} \rVert^2 + 2\lVert g_{td+\ell}^{(\ell)} \rVert^2. 
\end{aligned}
\label{eq:cyclic_eq}
\end{equation}

Then, summing $\ell=1,\dots,d$ and using~\eqref{eq:cyclic_eq}, we have

\begin{equation}
\begin{aligned}
    \lVert g_{td} \rVert^2 & = \sum_{\ell=1}^{d} \lVert g_{td}^{(\ell)} \rVert^2 \leq 2  \sum_{\ell=1}^{d} \left(  1 + (d-\ell)\frac{L^2}{\varsigma^2} \right) \lVert g_{td+\ell}^{(\ell)} \rVert^2.\\
    &\leq  2 \left(  1 + d\frac{L^2}{\varsigma^2} \right) \sum_{\ell=1}^{d} \lVert g_{td+\ell}^{(\ell)} \rVert^2.
\end{aligned}
\end{equation}

Summing over the number of cycles $t=0,1,\dots,T$, and since summing both for $t=0,1,\dots,T$ and for $\ell=0,1,\dots,d$ is equivalent to summing for $j=0,1,\dots,k$, then
\begin{equation}
     \sum_{t=0}^T \lVert g_{td} \rVert^2 
     \leq 2  \left(  1 + d\frac{L^2}{\varsigma^2} \right) \sum_{t=0}^T \sum_{\ell=1}^d \lVert g_{td+\ell}^{(\ell)} \rVert^2 = 2  \left(  1 + d\frac{L^2}{\varsigma^2} \right) \sum_{j=0}^k \lVert g_{j}^{(\ell_j)} \rVert^2,
     \label{eq:final_cyclic}
\end{equation}

which, taking the square root and using~\eqref{eq:square_to_norm}, gives
\begin{equation}
   \frac{1}{T+1} \sum_{t=0}^T \lVert g_{td} \rVert \leq \frac{1}{\sqrt{T+1}} \sqrt{\sum_{t=0}^T \lVert g_{td} \rVert^2} \leq \frac{1}{\sqrt{T+1}} \sqrt{2  \left(  1 + d\frac{L^2}{\varsigma^2} \right)\sum_{j=0}^k \lVert g_{j}^{(\ell_j)} \rVert^2}.
   \label{eq:bound_cyclic}
\end{equation}
Finally, combining~\eqref{eq:bound_cyclic} with~\eqref{eq:bound_glj}, we obtain

\begin{equation}
    \average_{t=0,\dots,T} \lVert g_{td} \rVert \leq   \sqrt{2  \left(  1 + d\frac{L^2}{\varsigma^2} \right)} \cdot  \frac{\Theta}{\sqrt{T+1}},
\label{eq:bound_1_cyclic}
\end{equation}
which is the thesis in~\eqref{eq:rate_cyclic}. 

\textbf{UR:} We now prove the rate of convergence of the \name\ algorithm when the UR rule is used.

Starting from~\eqref{gen-decr-det} and AS.3, summing for $j=0,\dots,k$, taking the expectation

\begin{equation}
    f(x_0)-\E \left[ \E_k \left[f(x_{k+1}) \right] \right]\geq 
\E \left[ \sum_{j=0}^k \E_j \left[\sum_{i=1}^N \frac{(g_{i,j}^{(\ell_j)})^2}{ w_{i,j}} \right] \right]
   - \frac{\Bar{L}}{2} \E \left[ \sum_{j=0}^k \E_j \left[ \sum_{i=1}^N \frac{ (g_{i,j}^{(\ell_j)})^2}{w_{i,j}^2} \right] \right],
   \label{gen-dec-k-det}
\end{equation}

 using the tower property, equation~\eqref{gen-dec-k-det} simplifies to

\begin{equation}
    \E\left[\sum_{j=0}^k\sum_{i =1}^N \frac{ (g_{i,j}^{(\ell_j)})^2}{ w_{i,j}} \right] \leq f(x_0)-f_{\text{low}}+\frac{\Bar{L}}{2} \E\left[ \sum_{j=0}^k \sum_{i =1}^N \frac{ (g_{i,j}^{(\ell_j)})^2}{ w_{i,j}^2} \right].
    \label{eq:start_stoch}
\end{equation}

We provide an upper bound for the right hand side in~\eqref{eq:start_stoch} by using the reformulation of $w_{i,j}$ in~\eqref{eq:explicit_weights} and Lemma \ref{gen:series} with $c_{j}=(g_{i,j}^{(\ell_j)})^2$. This bound is derived similarly to the one in~\eqref{use_lemma_det}, that is

\beqn{final_grad}
\begin{aligned}
 \E\left[\sum_{j=0}^k\sum_{i =1}^N \frac{ (g_{i,j}^{(\ell_j)})^2}{ w_{i,j}} \right] & \leq  \Gamma_0+ \frac{N \Bar{L}}{2}  \E \left[ \log\left( 1 + \frac{1}{\varsigma}  \sum_{j=0}^k \|g_j^{(\ell_j)}\|^2 \right) \right],\\
 & = \Gamma_0+ N \Bar{L}  \E \left[ \log\left( \sqrt{ 1 + \frac{1}{\varsigma}  \sum_{j=0}^k \|g_j^{(\ell_j)}\|^2 } \right) \right],\\
 & \leq \Gamma_0+ N \Bar{L} \log\left( \E \left[  \sqrt{ 1 + \frac{1}{\varsigma}  \sum_{j=0}^k \|g_j^{(\ell_j)}\|^2 }  \right] \right) ,\\
 \end{aligned}
\eeqn
where we used the concavity of the logarithm and Jensen inequality to derive the last bound.

Using Lemma~\ref{lem:first_sum} combined with~\eqref{final_grad},
\begin{equation*}
   \sqrt{\varsigma} \E \left[  \sqrt{ 1 + \frac{1}{\varsigma}  \sum_{j=0}^k \|g_j^{(\ell_j)}\|^2 }  \right] \leq \sqrt{\varsigma} +  \Gamma_0+ N \Bar{L} \log\left( \E \left[  \sqrt{ 1 + \frac{1}{\varsigma}  \sum_{j=0}^k \|g_j^{(\ell_j)}\|^2 }  \right] \right).
\end{equation*}

Similarly to the GS and cyclic case we can apply Lemma~\ref{log_bound} with 
$$ u= \E \left[  \sqrt{ 1 + \frac{1}{\varsigma}  \sum_{j=0}^k \|g_j^{(\ell_j)}\|^2 }  \right], \quad \gamma_1= \sqrt{\varsigma}, \quad \gamma_2=N \Bar{L}, \quad \gamma_3= \sqrt{\varsigma} +  \Gamma_0, $$
leading to 
\begin{equation}
    \E \left[  \sqrt{ 1 + \frac{1}{\varsigma}  \sum_{j=0}^k \|g_j^{(\ell_j)}\|^2 }  \right] \leq \frac{2(\sqrt{\varsigma}+\Gamma_0)}{\sqrt{\varsigma}}+\frac{2N\Bar{L}}{\sqrt{\varsigma}}\left[ \log\left( \frac{2N\Bar{L}}{\sqrt{\varsigma}} \right) -1\right],
\end{equation}

thus resulting in
\begin{equation}
    \E \left[ \sqrt{ \sum_{j=0}^k \lVert g_{j}^{(\ell_j)} \rVert^2} \right] \leq  \sqrt{\varsigma} \E \left[  \sqrt{ 1 + \frac{1}{\varsigma}  \sum_{j=0}^k \|g_j^{(\ell_j)}\|^2 }  \right] \leq 2(\sqrt{\varsigma}+\Gamma_0) + 2N\Bar{L}\left[ \log\left( \frac{2N\Bar{L}}{\sqrt{\varsigma}} \right) -1\right]=\Theta.
\end{equation}

Using~\eqref{eq:square_to_norm} again 
\begin{equation}
    \E\left[ \average_{j \in \{0,\dots,k\}} \lVert g_j^{(\ell_j)} \rVert\right] \leq  \E\left[ \sqrt{\average_{j \in \{0,\dots,k\}} \lVert g_j^{(\ell_j)} \rVert^2}\right] \leq \frac{ \Theta}{\sqrt{k+1}}.
    \label{eq:g_gR_stoch}
\end{equation}
Now we need to derive a bound for the norm of the whole gradient $g_j$.
Therefore, using~\eqref{eq:unbias} and the tower property, we have
\begin{equation}
\begin{aligned}
     \E\left[ \average_{j \in \{0,\dots,k\}} \lVert g_j^{(\ell_j)} \rVert\right]& =\frac{1}{k+1}\sum_{j=0}^k \E[ \E_j[\lVert g_j^{(\ell_j)} \rVert]] = \frac{1}{d(k+1)}\sum_{j=0}^k \E \left[ \sum_{\ell=1}^d \lVert g_j^{(\ell)} \rVert \right] \\
     & \geq  \frac{1}{d(k+1)}\sum_{j=0}^k \E \left[ \left\lVert  \sum_{\ell=1}^d g_j^{(\ell)} \right\rVert \right]  = \frac{1}{d(k+1)}\sum_{j=0}^k \E \left[ \lVert g_j \rVert \right]  = \frac{1}{d} \E \left[ \average_{j \in \{0,\dots,k\}} \lVert g_j \rVert  \right]
\end{aligned}
\label{eq:final_ineq_gj}
\end{equation}
Finally, combining~\eqref{eq:g_gR_stoch} with~\eqref{eq:final_ineq_gj}, we obtain~\eqref{gradbound}, that is, 
\begin{equation*}
     \E \left[ \average_{j \in \{0,\dots,k\}} \lVert g_j \rVert  \right] \leq  \frac{d \Theta}{\sqrt{k+1}}.
\end{equation*}

hence concluding the proof.
}

Let us comment the convergence of \name. 

\begin{enumerate}  
\item Similarly to the original AdaGrad algorithm, the nonmonotone behavior of \name\ can be deduced from Lemma~\ref{lemma:dec_bAdaGrad}. In fact, if the second summation in~\eqref{gen-decr-det-gs-c} or~\eqref{gen-decr-det} is larger than the first, then the function increases from iteration $j$ to iteration $j+1$. However, since the second term quadratically decreases with the optimization weights $w_{i,j}$, while the first term depends linearly on the weights, we expect the first to dominate the second in later stages of the algorithm. This behavior is consistently different from the majority of BCG methods in the literature, where constant step sizes or line-search are used. In fact, they usually imply a sufficient decrease of the function and result in monotonic algorithms. The nonmonotone behavior makes \name\ very flexible but it also results in weaker convergence guarantees, see the discussion below for more details.

\item While both the rates for random and greedy BCG depend on the total number of iterations $k$, the rate of cyclic BCG is given in terms of full number of cycles $T$ after which all the variables are updated. One can deduce that in a situation where the number of blocks $d$ is large, $T$ grows slower than $k$. Thus, our result for cyclic BCD is weaker than those for random and greedy BCG. Nevertheless, proving convergence rates depending on the full cycles $T$ is common practice when the cyclic block selection rule is used, see for example~\cite{beck2013convergence,cai2023cyclic}. 

\item Our bounds have an explicit dependence  on the number of blocks $d$. Ideally, one would like to remove it since it impacts the comparison between BCG and standard gradient method in favor of non-block approach. One possible strategy in random BCG, it has been shown that choosing the block according to a probability $p_\ell$ depending on the block Lipschitz constant $L_\ell$ achieves the goal~\cite{nesterov2012efficiency,allen2016even,nesterov2017efficiency}. However, in our setting, the Lipschitz constants are considered to be unknown. Alternatively, on might consider introducing alternative Lipschitz condition such as that proposed in~\cite{song2023cyclic,cai2023cyclic,wei2026adaptive} to achieve this goal for the C rule. We regard exploring this possibility in the \name\ framework as future research.

\item The constant $\Theta$ that appears in our bound also depends on the variable dimension $N$. This dependency is due to the coordinatewise structure of \name\ which, at each iteration, adaptively chooses a different step size for the update of each component in the selected block. In particular, the constant $N$ comes from the strategy we used in~\eqref{use_lemma_det} to bound the sum of quadratic terms in the descent lemma. We believe that the factor $N$ in the estimate might be improved, at least to the size of the largest block, but it is not clear how to achieve this goal. 
   
\end{enumerate}

 \subsection{Comparison with previous works}
 \label{sub:comp_works}
The first difference between this work and those closely related in Table~\ref{tab:rel_work_nonconvex} is that we provide bounds on the average gradient norm (ergodic rate) that are slightly stronger than those in the minimum gradient norm typically established in the contributions investigating similar smooth and nonconvex settings.
Nevertheless, convergence rates expressed in terms of the minimum or the average norm are of a similar nature and constitute a standard technique in nonconvex optimization and ergodic results are widely used in the analysis of AdaGrad-type algorithms, see~\cite{gratton2022parametric,gratton2024complexity,gratton2025complexity,porcelli2025prunadag}. We also express our bounds in term of the gradient norm instead of the gradient norm square as done in the majority of the contributions in Table~\ref{tab:rel_work_nonconvex}. To favor the comparison between the convergence rates analyzed in this section, we rewrite all the bounds in terms of the gradient norms.

Moreover, no sufficient decrease is imposed in our setting, and the block Lipschitz constants are assumed to be unknown, as it occurs in several practical applications. This yields larger constants in the convergence rate for \name\ with respect to BCG where the step size is constant or computed by a line-search. Nevertheless, \name\ remains more flexible since it does not use the block-Lipschitz constants or any function evaluations to choose the step size at each iteration. 

We now compare the rate of \name\ with those in the literature in more detail. Starting from randomized BCG, our result is similar to that proposed by Patrascu and Necoara~\cite[Corollary~1]{patrascu2015efficient} for constant step size BCG, that is, \begin{equation}
    \min_{j \in \{0,\dots,k\}}  \lVert g_j \rVert_L^*  \leq \sqrt{ \frac{2d (f(x^0)-f^*)}{k+1}}, \quad \text{where } \quad \lVert g_j \rVert_L^*=\sqrt{\sum_{\ell=1}^d \frac{\lVert g_j^{(\ell)} \rVert^2}{L_\ell}}, 
    \label{eq:pat_nec}
\end{equation} 
and $f^*$ denotes the optimal function value. Their constant in the bound is smaller than ours, and they prove a $\sqrt{d}$ dependence on the number of blocks $d$, while our rate involves a larger factor  $d$. For greedy BCG, Nutini and co-authors~\cite{nutini2015coordinate} consider a similar setting to the one analyzed in this paper. Their method uses an update, referred to as matrix update, of the form $x^{k+1}=x^k- H_{\ell_k}^{-1} g_k^{(\ell_k)}$, where $H_k$ is a positive definite matrix that represents an upper bound for the Hessian of $f$ in the twice differentiable case. However, constructing $H_{\ell_k}$ is not trivial in general and is problem dependent, making this approach less flexible. Nevertheless, convergence is established as:
\begin{equation}
    \min_{j \in \{0,\dots,k\}}\max_{\ell=1,\dots,d} \lVert g_j^{\ell} \rVert_{H_\ell^{-1}} \leq \sqrt{2 \frac{f(x^0)-f^*}{k+1}}, \quad \text{where } \quad \lVert \cdot \rVert_H=\sqrt{ \langle H \cdot, \cdot \rangle}.
    \label{eq:nutini_bound}
\end{equation}
Despite considering the standard Euclidean norm, the bound in this paper is similar to the one in~\eqref{eq:nutini_bound}. In fact, both the bounds present a linear dependence on the number of blocks $d$ and using the fact that $\lVert g_j \rVert^2 \leq d \max_{\ell=1,\dots,d} \lVert g_j^{\ell} \rVert^2$ we can deduce from~\eqref{eq:nutini_bound} a bound on the full gradient norm, similar to the one proposed in Theorem~\ref{theorem:bAdag_GS}. Nevertheless, the constant in our bound still remains larger since it involves an additional term due to the nonmonotone behavior of the algorithm.

In addition, for cyclic BCG, the contribution that is closer to our nonconvex setting is that by Cai et al.~\cite{cai2023cyclic}. Specifically, they analyze a composite (smooth + nonsmooth) minimization problem, where the smooth term is assumed to be a finite sum of functions and the nonsmooth term is convex and separable. Therefore, their convergence results are given in terms of the distance of the subgradient from zero. We readapt it here to the smooth case, obtaining:
$$\min_{j \in \{0,\dots,k\}}  \lVert g_j \rVert \leq \sqrt{4(1+ \Hat{L}) \frac{(f(x^0)-f^*)}{T+1}},$$ 
where $\Hat{L}$ is a generalized notion of a global Lipschitz constant with respect to a Mahalanobis
norm; we refer to~\cite{cai2023cyclic} for more details. Contrary to our bound, the constant of Cai et al. has the advantage of being independent of the number of blocks. Nevertheless, when only smooth functions are considered, their finite sum setting is more restrictive than ours which includes general convex functions. We also mention that the dependence on the number of blocks $d$ in our bound is similar to that proposed by Beck and Tetruashvili in~\cite{beck2013convergence} for convex, smooth functions. 

Finally, the \name\ algorithm includes the original AdaGrad as a special case when the number of blocks $d$ is equal to 1. In fact, for greedy and random BCG, setting $d=1$ in the convergence rates we recover a similar bound to that of the AdaGrad-like method presented in~\cite{bellavia2025fast}. For cyclic BCG, our result is weaker since for $d=1$, our bound has an additional factor $1+L/\varsigma$ which does not appear in the original method. The same observation is discussed in~\cite{beck2013convergence}, comparing their cyclic BCG with the standard gradient method in the convex case.  

\section{Extension to box-constrained problems}
\label{sec:constrained}
We now extend \name\ to box constrained optimization. Thus, we consider the following problem:
\begin{equation}
    \min_x f(x) \quad \text{subject to} \quad x \in C=\{ x \in \mathbb{R}^N \ | \ a_i \leq x_i \leq u_i \},
    \label{eq:const_prob}
\end{equation}
where $a_i, u_i \in \mathbb{R}\cup \{-\infty; \infty\}$. Therefore, when $a_i=-\infty$ and $u_i=\infty$, the unconstrained setting is recovered as a special case. We also denote $P_C$ the projection into the set $C$, which is easy to compute since $[P_C(x)]_i=P_{[a_i,u_i]}(x_i)=\max(a_i,\min(u_i,x_i))$. In what follows, we adapt the \name\ algorithm to this more general setting, extending the idea of Bellavia et al. in~\cite{bellavia2025fast}, where they consider stochastic AdaGrad for box-constrained problems.

\newpage
Let us introduce the box-constrained version of \name\ in Algorithm~\ref{alg:bAdaGrad_b}.
\algo{alg:bAdaGrad_b}{\tal{\namec}}
{
\begin{description}
\item[Step 0: Initialization. ]
  Bound constrained vectors $a, u \in \mathbb{R}^N$, a starting point $x_0 \in C$, a constant $\varsigma\in(0,1)$, and the initial weight vector $w_{i,-1}=\sqrt{\varsigma}$ for $i=1,\dots,N$, a fixed partition $\{1,\dots,N\} = \bigcup_{\ell=1}^d b_\ell$. Set $k=0$. 

\item[Step 1: Choose the block.] Select the block $b_{\ell_k} \subseteq \{1,\dots,N\}$ to update according to the GS, UR or C rule.
 \item[Step 2: Gradient] If not already computed in Step 1, compute the gradient $g_{i,k}$ for $i \in b_{\ell_k}$.
 \item[Step 3: Measure of optimality.  ] Compute
 \begin{equation*}
 \begin{aligned}
     &v_{i,k}=P_{[a_i,u_i]}(x_{i,k}-g_{i,k})-x_{i,k}\quad \quad (i \in b_{\ell_k}),\\
     &v_{i,k}=0 \quad \quad (i \in \{1,\dots,N\} \backslash b_{\ell_k}). 
 \end{aligned}
 \end{equation*}
  \item[Step 4: Optimization weights.] Compute 
  \beqn{weig_R_det_b}
  \begin{aligned}
      & \eta_{i,k} = \sqrt{(\eta_{i,k-1})^2+v_{i,k}^2}
      \quad \quad (i \in b_{\ell_k}),\\
      & \eta_{i,k}=\eta_{i,k-1} \quad \text{ and } \quad \quad (i \in \{1,\dots,N\} \backslash b_{\ell_k}).
  \end{aligned}
   \eeqn
   

\item[Step 6: Compute the step.] Set 
  \beqn{step_R_det_b}
  \begin{aligned}
       s_{i,k}&=\frac{ v_{i,k}}{\max(1,\eta_{i,k})}  \quad \quad (i \in b_{\ell_k}),\\
       s_{i,k}&=0 \quad \quad (i \in \{1,\dots,N\} \backslash b_{\ell_k}).
  \end{aligned}
  \eeqn


\item[Step 7: New iterate.] Define
   \begin{equation*}
       \qquad x_{k+1} = x_k + s_k, 
       \label{eq:iter_update_b}
   \end{equation*}  
    increment $k$ by one and return to Step~1.
\end{description}
}
Let us first introduce the direction we use for the update of the $i$-th component of the $k$th-iterate:
\begin{equation*}
     v_{i,k} = P_{[a_i,u_i]}(x_{i,k}-g_{i,k})-x_{i,k}, \quad \text{ for } \quad i\in b_{\ell_k}.
\end{equation*}
The vector $v_k$ is also used to update the new adaptive weights $\eta_k$ of the algorithm. In fact, $\lVert v_k \rVert$ is a standard optimality measure in constrained, first order optimization, see~\cite{conn2000trust}. Of course, if there are no constraints ($C=\mathbb{R}^N$), then $\lVert v_k \rVert=\lVert g_k \rVert$. In particular, this change in the optimality measure also affects the GS-type block selection rule in the box-constrained case, which is still denoted Gauss-Southwell for simplicity, and, at iteration $k$, selects the block $b_{\ell_k}$ associated to the largest norm $\lVert v_k^{(\ell)}\rVert$ for $\ell=1,\dots,d$.

We observe that the updates in Algorithm~\ref{alg:bAdaGrad_b} remain feasible within the iterations. We show it by distinguishing between two cases based on the value of $s_{i,k}$ in~\eqref{step_R_det_b}. If $s_{i,k}=v_{i,k}$,  it holds for every $i \in b_{\ell_k}$ that $x_{i,k+1}=P_{[a_i,u_i]}(x_{i,k}-g_{i,k})$ which is trivially feasible. On the other hand, if $s_{i,k}=v_{i,k}/\eta_{i,k}$ with $\eta_{i,k}>1$, then
\begin{equation*}
    x_{i,k+1}=x_{i,k}+\frac{v_{i,k}}{\eta_{i,k}}=\left(1-\frac{1}{\eta_{i,k}}\right)x_{i,k}+\frac{1}{\eta_{i,k}}P_{[a_i,u_i]}(x_{i,k}-g_{i,k}), \qquad \text{for }i \in b_{\ell_k},
\end{equation*}
which is a convex combination of points in the interval $[a_i,u_i]$, provided that the first iterate $x_0$ is in the feasible set $C$; thus, it is feasible. We used the property of box constraint to be convex in each component, which is crucial due to the coordinatewise nature of \namec\ algorithm. 

Moreover, we clarify that when $a_i=-\infty$ and $u_i=\infty$, Algorithm~\ref{alg:bAdaGrad_b} results in a slightly more conservative version of the \name\ algorithm, where the step size is set to the smaller value between 1 and $1/\eta_{i,k}$ for every $i \in b_{\ell_k}.$ Nevertheless, as the discussion above confirms, the modification in the choice of the step size in~\eqref{step_R_det_b} is necessary to ensure the feasibility of the iterate in the box-constrained case.

We now introduce the descent lemma for Algorithm~\ref{alg:bAdaGrad_b}. Similarly to Lemma~\ref{lemma:dec_bAdaGrad}, we prove the result just for the UR case since those for GS and C rule follow by removing the expectation in all the computations. 

\begin{lemma}
    \label{lemma:dec_bAdaGrad_b}
    Suppose that AS.1 - AS.3  hold. If Algorithm~\ref{alg:bAdaGrad_b} is applied to problem~\eqref{eq:const_prob}, then, according to the block selection rules GS, UR, and C, we have that, for all $j\ge0$,
\begin{equation}
  \text{\textbf{(GS/C) :}} \quad    f(x_{j+1})
\le f(x_j) -\varsigma \sum_{i=1}^N \frac{ (v_{i,j}^{(\ell_j)})^2}{ \eta_{i,j}}
     + \frac{L_{\ell_j}}{2} \sum_{i=1}^N \frac{ (v_{i,j}^{(\ell_j)})^2}{\eta_{i,j}^2}, 
      \label{eq:decr_f_det_gs_c} 
\end{equation}
\begin{equation}
    \text{\textbf{(UR) :}} \quad  \E_j\left[f(x_{j+1})\right]
\leq f(x_j) -\varsigma \E_j\left[\sum_{i=1}^N \frac{ (v_{i,j}^{(\ell_j)})^2}{ \eta_{i,j}} \right]
      +\frac{L_{\ell_j}}{2}  \E_j\left[ \sum_{i=1}^N \frac{ (v_{i,j}^{(\ell_j)})^2}{\eta_{i,j}^2} \right]. 
    \label{gen_decr_ur}
\end{equation}
\end{lemma}

\proof{
Let us first observe that the non-expansiveness of the projection operator yields the inequality $  |g_{i,j}| \geq |v_{i,j}|$. Therefore, we distinguish between two disjoint cases, based on the value of the step $s_{i,k}$. In the first case ($s_{i,k}=v_{i,k}$), we have for every $i \in b_{\ell_k}$
\begin{equation}
    \lvert g_{i,j} s_{i,j} \rvert \geq  v_{i,j}^2= \eta_{i,j} \frac{v_{i,j}^2}{\eta_{i,j}}\geq \varsigma \frac{v_{i,j}^2}{\eta_{i,j}}.
    \label{eq:first_step_b}
\end{equation}
In the second case ($s_{i,k}=v_{i,k}/\eta_{i,j}$), for every $i \in b_{\ell_k}$, it holds $\lvert g_{i,j} s_{i,j} \rvert \geq  \frac{v_{i,j}^2}{ \eta_{i,j}}$. Thus, recalling that $\varsigma \in (0,1]$, and taking the minimum with~\eqref{eq:first_step_b}, we get $ \lvert g_{i,j} s_{i,j} \rvert \geq \varsigma(  v_{i,j}^2/\eta_{i,j}).$ Since, by construction, $g_{i,j} s_{i,j} \leq 0$,  we finally have
\begin{equation}
     g_{i,j} s_{i,j}  \leq - \varsigma \frac{  v_{i,j}^2}{\eta_{i,j}}.
    \label{eq:second_step_b}
\end{equation}
 
In addition, it trivially follows from~\eqref{weig_R_det_b} and~\eqref{step_R_det_b} that $s_{i,k}^2 \leq v_{i,k}^2/\eta_{i,k}^2$ 

Using AS.3 and the observation in~\eqref{eq:taylor_app}

\begin{equation*}
\begin{aligned}
\E_j\left[f(x_{j+1}) \right]&\leq  f(x_j)+ \E_j\left[ g_j^T s_j \right]+  \frac{L_{\ell_j}}{2} \E_j\left[ ||s_j||^2 \right] \\
&= f(x_j)+ \E_j\left[ \sum_{i\in b_{\ell_j}} g_{i,j}s_{i,j} \right]+  \frac{L_{\ell_j}}{2} \E_j\left[ \sum_{i\in b_{\ell_j}}s_{i,j}^2 \right] \\
&\leq f(x_j) - \varsigma \E_j\left[ \sum_{i \in b_{\ell_j}} \frac{ v_{i,j}^2}{ \eta_{i,j}} \right]
      +\frac{L_{\ell_j}}{2} \E_j\left[ \sum_{i \in b_{\ell_j}} \frac{ v_{i,j}^2}{\eta_{i,j}^2} \right]\\
&=f(x_j) - \varsigma \E_j\left[ \sum_{i = 1}^N \frac{ (v_{i,j}^{(\ell_j)})^2}{ \eta_{i,j}} \right]
     + \frac{L_{\ell_j}}{2} \E_j\left[ \sum_{i=1}^N \frac{ (v_{i,j}^{(\ell_j)})^2}{ \eta_{i,j}^2}      \right],
\end{aligned}
\end{equation*}
 proving the lemma.
}

We now present our non-asymptotic convergence result for Algorithm~\ref{alg:bAdaGrad_b}, covering the GS, UR and C block selection rules. The proof is omitted since it follows directly from the new descent Lemma~\ref{lemma:dec_bAdaGrad_b}, by applying the same strategy used in Theorem~\ref{theorem:bAdag_GS}. The only difference is that the bounds are derived in terms of the suitable optimality measure $\lVert v_j \rVert$ for the constrained setting, rather than the standard gradient norm. While the extension for the GS and UR rules is straightforward, we show in the Appendix (Lemma~\ref{lem:append_bound}) how to derive an upper bound on the average norm of $v_j$ in terms of the average norm of $v_j^{(\ell_j)}$, when the C rule is employed.

\begin{theorem}
    \label{theorem:bAdag_GS_b}
    Suppose that AS.1--AS.3 hold and that the
Algorithm~\ref{alg:bAdaGrad_b} is applied to problem \req{eq:const_prob}.
Let $\varsigma>0$ and $d$ the number of blocks. Defining
$\Gamma_0 \eqdef f(x_0)-\flow,$
\begin{itemize}
    \item \textbf{GS:}
    \begin{equation}
        \average_{j\in\iiz{k}}\|v_j\| \le \sqrt{d} \frac{\Theta_c }{\sqrt{k+1}}.
        \label{gradbound_det_b}
    \end{equation}
    \item \textbf{C:} if, in addition, AS.4 holds, then
    \begin{equation}
        \average_{t \in \{0,\dots,T\}} \lVert v_{td} \rVert \leq  \sqrt{2  \left(  1 + d\frac{6+4L^2}{\varsigma^2} \right)} \frac{\Theta_c }{\sqrt{T+1}}.
        \label{eq:rate_cyclic_b}
    \end{equation}
    \item \textbf{UR:}
    \begin{equation}
        \E \left[ \average_{j\in\iiz{k}}\|v_j\| \right] \le d \frac{  \Theta_c}{\sqrt{k+1}},
        \label{gradbound_b}
    \end{equation}
\end{itemize}

with \begin{equation}
    \label{k6-def_b}
     \Theta_c \eqdef \frac{2(\Gamma_0 +\varsigma^{\frac{3}{2}})}{\varsigma} + \frac{2N\Bar{L}}{\varsigma}\left[ \log\left( \frac{2N\Bar{L}}{\varsigma^{3/2}} \right) -1\right],
\end{equation} 

where $\Bar{L}=\max_{\ell=1,\dots,d}L_\ell$.
\end{theorem}

\section{Numerical experiments}  
\label{sec:exp} 

In this section, we present numerical results on synthetic and real-world optimization problems, both in convex and nonconvex settings. Specifically, we analyze \name\ applied to convex least-squares and to nonconvex logistic regression with log-sum penalization and its box constrained version applied to robust Nonnegative Matrix Factorization (NMF) with the Huber function (Huber-NMF). 

In what follows, we also consider noisy objective functions and gradients. The random noise is added following the idea proposed in~\cite{gratton2024nonconvex}, that is, 
\begin{equation}
    f(x)=f(x)(1+\delta \mathcal{N}(0,1)), \quad \text{and} \quad [g(x)]_i=[g(x)]_i(1+\delta \mathcal{N}(0,1)),
    \label{eq:noise}
\end{equation}
where $\delta \in [0,1]$ controls the noise level and $\mathcal{N}(0,1)$ is the normal distribution. 
We aim at:
\begin{enumerate}
    \item comparing the performances of the three block selection strategies of \name\ across different applications;
    \item comparing the \name\ algorithm with BCG using line-search to choose the step size (BCG-LS) showing the advantages of using an OFFO algorithm, such as \name, in situations where the objective function and the gradient are affected by noise;
    \item showing the effectiveness of \name\ with bound constraint (Algorithm~\ref{alg:bAdaGrad_b}) in solving Huber-NMF. 
\end{enumerate}

In the least squares and logistic regression experiments, we compare our method with BCG-LS, where at each iteration we choose a step size satisfying the following Armijo condition:
\begin{equation}
    f(x_k+\alpha_k g_k^{(\ell_k)}) \leq f(x_k) - c_1 \alpha_k \lVert g_k^{(\ell_k)} \rVert^2,
    \label{eq:Armijo}
\end{equation}
where $c_1$ is set to $10^{-4}$.  The condition is checked using a standard backtracking strategy with a maximum of 20 inner iterations and a reduction of the step by a factor of 2 any time the Armijo condition is not satisfied. The first step at every iteration in the backtracking procedure is tuned for each class of problems. When dealing with noisy functions and gradients, there is no guarantee that the backtracking procedure in the BCG-LS algorithm produces an admissible step size; however, we do not report it as a failure but we take the last value given by the procedure and continue. 

\textbf{Setup} We set $\varsigma=0.0001$ in the \name\ implementations. We stop each algorithm either when a fixed time limit or maximum number of iterations is reached, or when the norm of the gradient is below a fixed threshold depending on the experiment; the stopping criteria are specified in each experiment. Since we do not want to evaluate the full gradient at each iteration, we check the convergence criteria on the full gradient norm every 20 iterations . All the results displayed are averaged over 10 runs. The algorithms are implemented in MATLAB  R2021b on a  64-bit Samsung/Galaxy with 11th Gen Intel(R) Core(TM) i5-1135G7 @ 2.40GHz  and 8 GB of RAM, under Windows 11 version 23H2.

\subsection{Least squares} Let $A \in \mathbb{R}^{m\times n}$ and $b \in \mathbb{R}^m$, we solve least squares problems of the form
\begin{equation*}
    f(x)=\frac{1}{2}\lVert Ax-b \rVert_2^2.
\end{equation*}
For this experiment, we subdivide the variable $x \in \mathbb{R}^n$ into 10 blocks of dimension $\lfloor n/10 \rfloor$ or $\lceil n/10 \rceil$. In our experiments, we consider 14 different matrices $A$ divided into the following classes:
\begin{enumerate}
    \item Random matrices: randomly generated matrices from a standard normal distribution of dimensions $m \times n$ (\texttt{X = randn(m,n)} in MATLAB) with $(m,n) \in \{(1000,500),  (500,500), \allowbreak (500,1000)\}$.
    \item Random sparse matrices: randomly generated matrices from a standard normal distribution of dimensions with $(m,n) \in \{(1000,500),(500,500),(500,1000)\}$ where 50$\%$ of the entries are randomly replaced with zeros.
    \item Bernoulli matrices: random matrix of dimension $500 \times 1000$ where each entry is either 0 or 1 with the same probability.
    \item Discrete cosine transform matrix: we generate the matrix with \texttt{dctmtx(n)} in MATLAB and we truncate it to the first 500 rows corresponding to lowest frequencies, obtaining a wide matrix of size $500 \times 1000$.
    \item Sparco matrices: 6 matrices from the collection of sparse signal recovery problems in~\cite{van2007sparco} and supplied by S2MPJ~\cite{GratToin24}. Given a sparse vector $x^*$, the observation is generated as $b=Ax^*+r$, where $r $ is additive noise vector of appropriate dimension and $A$ is a fixed dictionary consisting of various bases such as wavelet, discrete cosine, and Fourier.
\end{enumerate}

According to the definition in~\eqref{eq:noise}, we apply random noise to our least square problems, considering four noise levels $(\delta=0, 0.2,0.3,0.4)$. We generate 10 different instances for each problem by randomly changing the noise and the starting point for the algorithms. Thus, we run a total of 140 experiments. We randomly initialize each algorithm from a norm one vector $x_0/||x_0||_2$, where $x_0=\texttt{ randn(n,1)}$ in MATLAB.  We set a tolerance on the norm of the gradient at $10^{-3}$ and a maximum time limit of 20 seconds. If the tolerance is not reached within the allotted time, we report the run as a failure. We consider performance profiles for assessing the quality of the solution, according to the definition in~\cite{dolan2002benchmarking}. We choose the weighted sum of gradient and function evaluations, denoted as $N_f$ and $N_g$ respectively, as performance measure, with the relation $N_g=(1+\frac{d_{\max}}{N}) N_f$, where $d_{\max}$ is the dimension of the largest block. 
This relation is due to the higher computational cost of computing the block gradient with respect to evaluating the objective function. Therefore, one iteration of \name\ with C or UR rule which requires one block gradient computation, has a cost per iteration of $(1+\frac{d_{\max}}{N}) N_f$; while the cost per iteration is $2N_f$ if the GS rule is used, since it requires a full gradient evaluation per iteration. For BCG-LS, the cost is $(1+\frac{d_{\max}}{N}) N_f+\mu N_f$, where $\mu$ is the number of inner cycle in the backtracking strategy. We then compute 
$$c_{p,s}=\text{cost required by solver $s$ to solve problem $p$},$$
and evaluate
\begin{equation}
    \rho(\tau)=\frac{1}{n_p} \left\lvert \left\{p \  : \  \frac{c_{p,s}}{\min_{s}c_{p,s}} \leq \tau \right\} \right\rvert,
    \label{eq:perf_ratio}
\end{equation}
where $n_p$ is the total number of problems considered.
Our performance profile shows the percentage $\rho(\tau)$ computed over all problems considered of a given solver to have a performance ratio (solver performance measure / best performance measure) within a factor $\tau$ of the best possible solver. The results are displayed in Figure~\ref{fig:perf_prof}. The initial step size in the backtracking procedure of the BCG-LS algorithm is set to 1.

\begin{figure}[h]
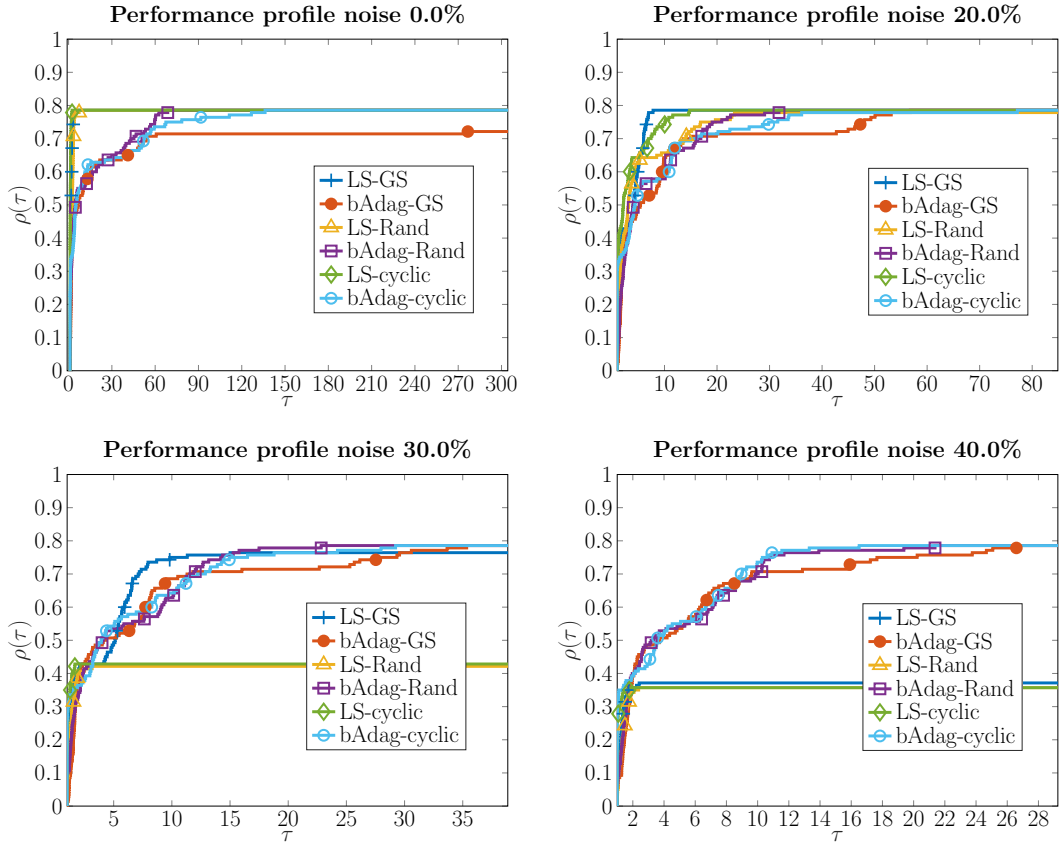

\centering
\begin{minipage}[h]{0.48\linewidth}
   \resizebox{1.05\textwidth}{!}{\input{Tikz/perf_ls_20s_0n}}
\end{minipage}
\begin{minipage}[h]{0.48\linewidth}
    \resizebox{1.05\textwidth}{!}{\input{Tikz/perf_ls_20s_20n}}
\end{minipage}

\begin{minipage}[h]{0.48\linewidth}
    \resizebox{1.05\textwidth}{!}{\input{Tikz/perf_ls_20s_30n}}
\end{minipage}
\begin{minipage}[h]{0.48\linewidth}
    \resizebox{1.05\textwidth}{!}{\input{Tikz/perf_ls_20s_40n}}
\end{minipage}
    \caption{Performance profile according the definition given in~\cite{dolan2002benchmarking}. The graphs show the percentage $\rho(\tau)$ in~\eqref{eq:perf_ratio} having a performance ratio within a factor $\tau$ from the best solver. Results on 140 instances of least squares problems. }
    \label{fig:perf_prof}
\end{figure}

The performance profiles show that, in the absence or for small values of noise, BCG-LS is the best algorithm for any block selection rule considered, being more than 90 times less expensive than \name\ in the noiseless case and more than 30 times less expensive with $20\%$ of noise. We observe an opposite behavior when the level of noise increases, in fact, the percentage of successfully solved problem of BCG-LS with random or cyclic block selection rule, decreases from $80\%$ to $40\%$ when the noise level is above the $30\%$. BCG-LS with GS rule still maintains a success rate over $70\%$ when the noise level is $30\%$, while the accuracy drops to $40\%$ when the noise is increased. On the other hand, \name\ is more robust, maintaining a constant success percentage in the allotted time of approximately $80\%$ for all noise levels. Among the different variants of \name, the cyclic rule outperforms the GS criteria for all noise levels considered.

\subsection{Logistic regression with log-sum penalty (LSP)}
\label{subsec:logistic}
We address a binary logistic regression problem with a smooth log-sum penalty (LSP) of the form
\begin{equation}
f(x)=\frac{1}{N_S} \sum_{j=1}^{N_S} \log(1+e^{-z_j y_j^Tx} )+\lambda \sum_{i=1}^N\log(1+\alpha x_i^2),
\label{eq:logistic}
\end{equation}
where  $\{y_j,z_j\}$ is a labeled dataset, and at each sample $  y_j \in \mathbb{R}^N$ corresponds a binary label $ z_j \in \{0,1\}$ for $j=1,\dots,N_S$ that classifies $y_j$ into two disjoint classes. The objective function in~\eqref{eq:logistic} is smooth and nonconvex due to the LSP term. The log-sum function was first proposed by Candes et al.~\cite{candes2008enhancing} as a relaxation of the sparsity-inducing $\ell_0$-norm and it can be used as an alternative regularizer to the $\ell_1$-norm in logistic regression that is less biased towards large coefficients~\cite{yuan2023feature}.   

In our experiments, we consider the following labeled data sets: MUSH~\cite{mushroom_73}, MNIST\footnote{Classification between even and odd numbers.}~\cite{LeCun2005TheMD}, GISETTE~\cite{machrep}, 
REGEDO~\cite{chang2011libsvm}, 
A9A~\cite{machrep}, 
and MOLECULE~\cite{machrep}. For the MNIST and A9A, datasets, a random subset of 1,000 samples is selected for the experiments while all remaining datasets are used in their entirety. The data samples are randomly divided into training and testing set with a ratio of 70:30. We subdivide the variable $x \in \mathbb{R}^n$ into 10 blocks of dimension $\lfloor n/10 \rfloor$ or $\lceil n/10 \rceil$. We minimize the logistic loss in~\eqref{eq:logistic} on the training set, fixing the parameters to $\lambda=0.1$ and $\alpha=10$.  We randomly initialize each algorithm from a norm one vector $x_0/||x_0||_2$, where $x_0=\texttt{ randn(n,1)}$ in MATLAB. We run each algorithm for 20 seconds and we stop it if a tolerance of $10^{-9}$ on the real gradient norm is reached. The initial step size in the backtracking procedure of the BCG-LS algorithm is chosen as $\lVert x_k \rVert/\lVert g_k^{(\ell_k)} \rVert$. We report in Table~\ref{tab:fun_comp} the final objective function value averaged over 10 random initializations, for different values of the noise level $\delta$. For completeness, we include in the supplementary materials the tables containing the final gradient norm (Table~\ref{tab:acc_comp}) and the percentage of correct classification on the testing set (Table~\ref{tab:res_comp}).

\begin{table}[h!]
\centering
\begin{tabular}{ c || c | c | c | c | c | c | c || c } 
 \toprule 
    &  \multirow{2}{*}{}&  \multicolumn{3}{c|}{\name} &  \multicolumn{3}{c||}{BCG-LS}\\
\hline
\textbf{Data set} & $\delta$ &  \hspace{-0.1cm}GS\hspace{-0.3cm}&\hspace{-0.1cm} UR\hspace{-0.3cm}&\hspace{-0.1cm}Cyclic\hspace{-0.3cm}&\hspace{-0.1cm}GS\hspace{-0.3cm}& \hspace{-0.1cm}UR\hspace{-0.3cm}&\hspace{-0.1cm}Cyclic\hspace{-0.3cm} \\
  \midrule
    \hspace{-0.10cm} Mush  \multirow{5}{*}{}& 0 & 0.1655   &  0.1655  & 0.1656 &   0.1653 &  0.1659  &  \textbf{0.1652}    \\
 &   0.10 & \textbf{0.1655} & \textbf{0.1655}  & 0.1656 &  0.1753  &	 1.8548 & 0.5943    \\
  &  0.20  & \textbf{0.1656}  &  \textbf{0.1656} &  0.1657 &  0.1874 &	13.49 &  13.97  \\
  &  0.30 & \textbf{0.1656}  &  \textbf{0.1656} & 0.1658 &  1.3503 &	Inf &  Inf  \\
   \hline
\hspace{-0.10cm} Gisette  \multirow{5}{*}{}& 0 & 0.0952   &  0.2555 &    0.1068&   0.0929 &  0.0945   & \textbf{0.0924}    \\
 &   0.10 & 0.1578  &  0.2641  & 0.1085  & \textbf{0.0989} &	 0.1510 &  0.1312  \\
  &  0.20  & \textbf{0.0971}  & 0.1658 & 0.1056 & 0.1228 &	  0.3250 &  0.2861  \\
  &  0.30 & 0.1388 & 0.6569  & 0.1335 & \textbf{0.1265}   & 31.01 &   15.77 \\
   \hline
\hspace{-0.10cm} MNIST  \multirow{5}{*}{}& 0 & \textbf{0.2952}  & \textbf{0.2952}  &   \textbf{0.2952}  &  \textbf{0.2952}&   \textbf{0.2952}  &   0.2954   \\
 &   0.10 & \textbf{0.2952}  & \textbf{0.2952}  &   \textbf{0.2952}  & 0.3143 &	0.6423 & 0.5726   \\
  &  0.20 & \textbf{0.2952}  & \textbf{0.2952}  &   \textbf{0.2952}  &  0.3221 & Inf & Inf   \\
  &  0.30 & \textbf{0.2952}  & \textbf{0.2952}  &   \textbf{0.2952}  &  0.4235 & Inf & Inf   \\
   \hline
   \hspace{-0.10cm} Regedo  \multirow{5}{*}{}& 0 & 0.1758   & \textbf{0.1752}   &   0.2431  &   0.1797 &   0.1810 &   0.1834  \\
 &   0.10 & \textbf{0.1748 }  & 0.1764   &   0.2360  &   0.1965 &   0.3001  &   0.2897  \\
  &  0.20  & 0.1782   & \textbf{0.1745}   &   0.2478  &   0.2575 &   8.872  &   1.378  \\
  &  0.30 & 0.1790   & \textbf{0.1754}   &   0.2137  &   0.2411 &   Inf  &   73.45  \\
   \hline
   \hspace{-0.10cm} A9A  \multirow{5}{*}{}& 0 &  \textbf{0.3930} & \textbf{0.3930}   &  \textbf{0.3930} &   \textbf{0.3930}  & \textbf{0.3930} &  \textbf{0.3930} \\
 &   0.10 &  \textbf{0.3930} & \textbf{0.3930} &  \textbf{0.3930}  &  0.4143  &   20.41  &   20.27  \\
  &  0.20  &  \textbf{0.3930} & \textbf{0.3930} &  \textbf{0.3930}  &  0.4367  &   Inf  &   Inf  \\
  &  0.30&  \textbf{0.3930} & \textbf{0.3930} &  \textbf{0.3930}  &  0.5822  &   Inf &   Inf \\
   \hline
 \bottomrule 
\end{tabular}
\caption{Final function value averaged over 10 random initializations for the logistic regression with LSP, for different levels of noise $\delta$ according to the definition in~\eqref{eq:noise}. Lowest function values per row are highlighted in bold and we denote as Inf if the simulation produced a non-finite function value at least once over 10 random initializations. } 
\label{tab:fun_comp}
\end{table}

Table~\ref{tab:fun_comp} highlights the robustness of \name\ with respect to BCG-LS. In fact, we observe that in the noiseless scenario, BCG-LS provides the lowest value in the objective function for 4 out of 5 data seta but it suffers a significant drop in accuracy even for small values of noise. On the contrary, \name\ is less affected by the noise as the unchanged value of the function for MUSH, MNIST, and A9A testifies. Moreover, the GS rule achieves better accuracy on average both for \name\ and BCG-LS.

\subsection{Nonnegative Matrix Factorization (NMF) with the Huber loss}

We consider a robust NMF problem using the Huber function as error measure. NMF aims at approximating a given nonnegative matrix $X\in \mathbb{R}_+^{m\times n}$ by the product of two nonnegative factors, $W \in \mathbb{R}^{m \times r}$ and $H \in \mathbb{R}^{r \times n}$, where $r \ll \min(m,n)$ is the rank of the decomposition. NMF model finds application in several fields, such as hyperspectral unmixing~\cite{bioucas2012hyperspectral,ma2013signal}, topic modeling and document classifications~\cite{shahnaz2006document,ding2008equivalence}, and feature extraction~\cite{lee1999learning,guillamet2002non}, see~\cite{gillis2020nonnegative} for a book on the topic.

When data are affected by sparse noise or outliers, the common NMF models using the Frobenius norm or the KL divergence as an error measure might suffer from poor performance. For this reason, robust NMF models arise as a possible alternative by considering robust error measures such as $\ell_1$-norm~\cite{ke2005robust,guan2012mahnmf,seraghiti2026nonnegative}, $\ell_{2,1}$-norm~\cite{kong2011robust}, or the Huber function~\cite{du2012robust,guo2021modified,barkhoda2026instance}. In this work, we focus on a regularized version of Huber NMF, that can be formulated as follows:
\begin{equation}
    \min_{W \geq0,H\geq0} \Psi_\rho \left(  X - WH \right) + \lVert W\rVert_F^2 + \lVert H\rVert_F^2 \quad \text{with} \quad \Psi_\rho(a) =
\begin{cases}
\frac{1}{2}a^2 & \text{if } |a| \le \rho \\
\rho \left(|a| - \frac{1}{2}\rho\right) & \text{if } |a| > \rho
\end{cases},
    \label{eq:huber}
\end{equation}
where the Huber function is applied componentwise. The objective function in~\eqref{eq:huber} is smooth, nonconvex, and its gradient is block-Lipschitz. However, due to the product $WH$, the function does not have a globally Lipschitz gradient. BCD algorithms are state-of-the-art techniques for solving matrix factorization problems, as they decompose the optimization into subproblems corresponding to each factor. In this setting, due to the presence of the Huber loss, exact BCD approaches do not admit closed-form updates for $W$ and $H$. Moreover, deriving explicit block-wise Lipschitz constants is not straightforward, which makes the choice of step size in BCG nontrivial. Even though the lack of the global Lipschitz property of the gradient prevents cyclic \name\ from having convergence guarantee,  cyclic block selection rule is the most common approach for BCG in the NMF context. Therefore, we choose the cyclic rule for the experiments in this section.

Despite NMF being a nonconvex problem, it is well-known that extrapolation helps accelerating the convergence of BCG in the context of matrix factorization~\cite{hien2025block,wen2012solving,ang2019accelerating}. Thus, we implemented an extrapolated version of \name, dubbed \ename, which used two extrapolation steps after the update of each factor, that are
\begin{equation*}
    W^{k+1}=\max(0,W^{k+1}+\beta (W^{k+1}-W^k)), \quad H^{k+1}=\max(0,H^{k+1}+\beta (H^{k+1}-H^k)),
\end{equation*}
with $\beta$ fixed to 0.9. We compare our algorithms \name\ and \ename\ with the Multiplicative Update (MU) algorithm proposed in~\cite{du2012robust}, which is a specialized algorithm for NMF.

For this experiment, we consider the CBCL dataset, containing 2429 vectorized greyscale facial images of dimension  $19 \times 19$, resulting in a data set $X_t \in \mathbb{R}^{2429 \times 361}$. We add sparse noise to the images by randomly setting $7 \%$ of the pixels in white, obtaining a noisy version of the data $\Hat{X}$. The rank of the factorization is $r=49$. The $\rho$ parameter in the Huber function is set to the mean value of the noisy dataset $\Hat{X}$ and we fix both $\lambda_W$ and $\lambda_H$ to $10^{-4}$. All the algorithms are randomly initialized: the entries of $W$ and
$H$ are sampled from a uniform distribution in $[0,1]$. We then scale them so that the initial point matches the magnitude of the original matrix, that is,
we multiply $W$ by $\sqrt{\lVert \Hat{X} \rVert_F}/\lVert W \rVert_F$ and $H$ by $\sqrt{\lVert \Hat{X} \rVert_F}/\lVert W \rVert_F$. For this experiments, we run each algorithm for 500 iterations and no other stopping criteria is considered. We show the decrease of the objective function in~\eqref{eq:huber}, normalized dividing by the norm of $X$ in Figure~\ref{fig:cbcl_comp}. Moreover, we display in Table~\ref{tab:CBCL_t} the final relative error in Frobenius norm with the original data ($\lVert X_t - WH\rVert_F/\lVert X_t \rVert_F$) and the CPU time for each algorithm.

\begin{figure}
    \centering
    \scalebox{0.9}{\input{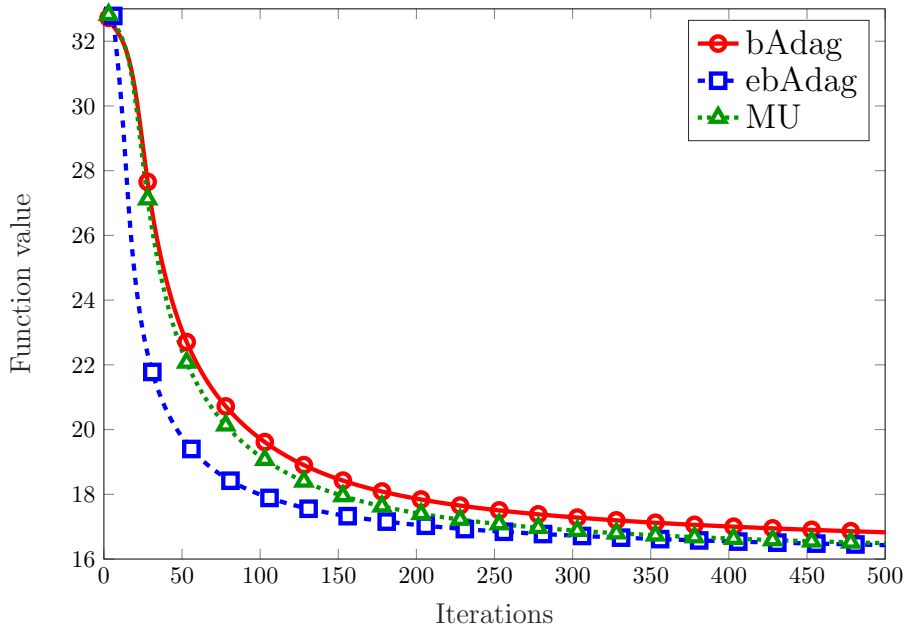}}
    \caption{Huber-NMF objective function values along the iterations. All the algorithms run for 500 iterations and results are averaged over 10 random initializations. }
    \label{fig:cbcl_comp}
\end{figure}

We observe from Figure~\ref{fig:cbcl_comp} that all algorithms have comparable results with the \ename\ having a faster decrease in the first stages and MU being slightly faster than \name. The final accuracy and running time in Table~\ref{tab:CBCL_t} confirms the similar behavior of \ename\ and MU that reach a better accuracy than \name.

\begin{table}[h]
\centering
\begin{tabular}{cccc}
\toprule
 & \name & \ename & MU \\
\midrule
Relative error  & 0.1684  & 0.1681  & \textbf{0.1675}  \\
CPU time  &  13.39 & 13.54  &  \textbf{12.72} \\

\bottomrule
\end{tabular}
\caption{Final relative error in Frobenius norm ($\lVert X_t - WH\rVert_F/\lVert X_t \rVert_F$) and CPU time for each algorithm to compute Huber-NMF on the CBCL data. Results are averaged over 10 random initializations. Lowest error and running time are highlighted in bold.}
\label{tab:CBCL_t}
\end{table}
We show an example of the images reconstructed by the \ename\ algorithm in Figure \ref{fig:reconstruction}. 

\begin{figure}[h]
    \centering
    \begin{minipage}[h]{0.32\linewidth}
        \includegraphics[width=\linewidth]{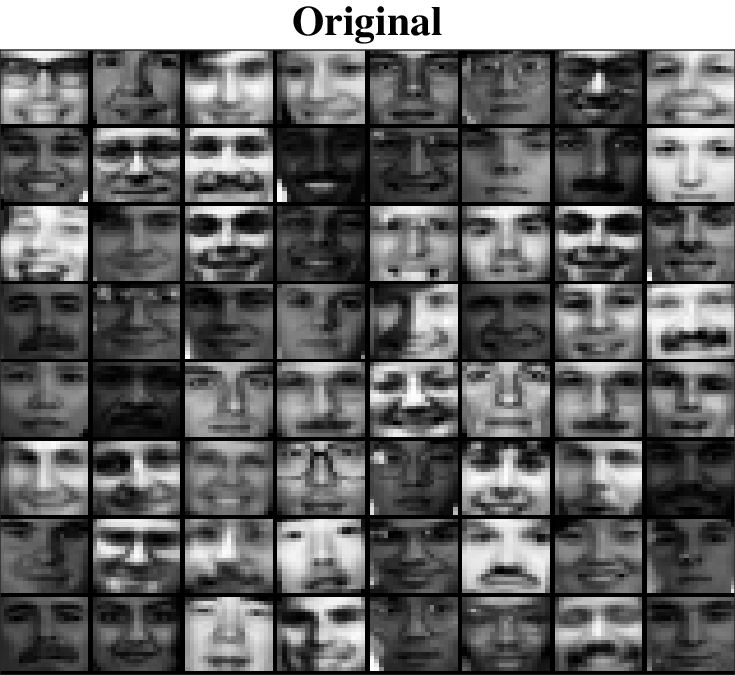}
    \end{minipage}
     \begin{minipage}[h]{0.32\linewidth}
        \includegraphics[width=\linewidth]{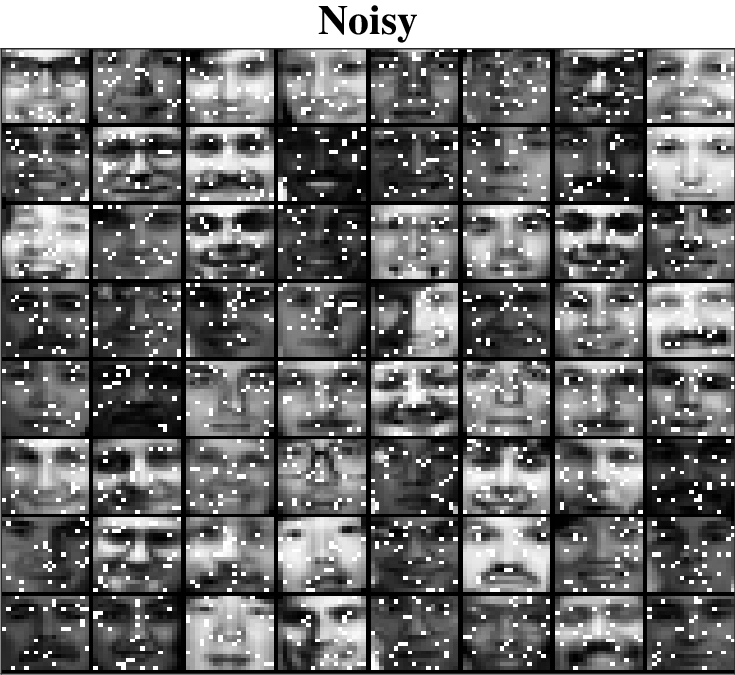}
    \end{minipage}
     \begin{minipage}[h]{0.32\linewidth}
        \includegraphics[width=\linewidth]{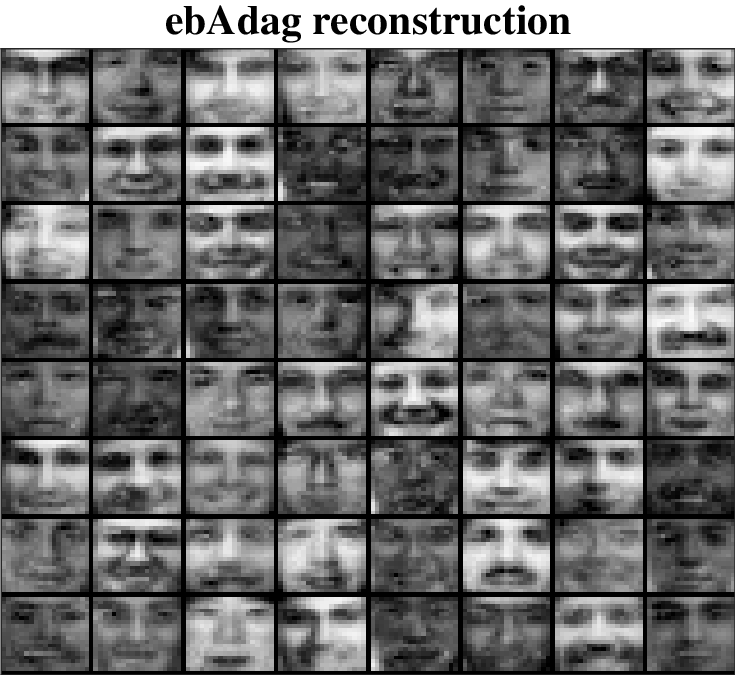}
    \end{minipage}
    \caption{Visual reconstruction of 64 randomly selected images from the CBCL data. Original images on the left, noisy images in the center with $7\%$ of randomly added white pixels, rank-49 Huber-NMF reconstruction after 500 iterations of the \ename\ algorithm.}
      \label{fig:reconstruction}
\end{figure}

\section{Conclusion}
In this paper, we proposed the \name\ algorithm: a novel adaptive block coordinate gradient method which extends the AdaGrad algorithm to the block coordinate context. We proved sublinear convergence rates for the three main block selection rules, namely Gauss-Southwell, uniform random selection, and cyclyc rule. To our knowledge this is the first adaptive algorithm covering the three block selection rule in the smooth, nonconvex setting. We also extended our framework and its convergence properties to box-constrained problems. We validated our algorithms on least squares and logistic regression, and on robust NMF.

\paragraph*{Acknowledgment} I thank Margherita Porcelli and all the NODA group (University of Florence) and Nicolas Gillis (University of Mons) for useful discussion and feedback on this work.


{\footnotesize
\bibliographystyle{plain}
\bibliography{bib.bib}
}

\appendix
\section{Appendix}
\subsection{Supplementary material to Section~\ref{sec:constrained}}
\begin{lemma}
\end{lemma}

    Suppose AS.1-AS.4 hold and that Algorithm~\ref{alg:bAdaGrad_b} is applied to problem~\eqref{eq:const_prob}. Then it holds 
    \begin{equation*}
        \average_{j \in \{0,\dots,k\}} \lVert v_j \rVert \leq \sqrt{2\left( 1+ d \frac{6+ 4L^2}{\varsigma^2} \right) \average_{j \in \{0,\dots,k\}}\lVert v_j^{(\ell_j)} \rVert^2}.
    \end{equation*}
    \label{lem:append_bound}

\proof{

Let us use the same reparametrization introduced in case 2 of the proof of Theorem~\ref{theorem:bAdag_GS}. Thus, we rewrite the iteration count as $k=Td+\ell$, where $\ell=0,\dots,d-1$. By the structure of Algorithm~\ref{alg:bAdaGrad_b},
\begin{equation}
    x_{td+\ell} =x_{td}- \sum_{s=1}^{\ell-1} \frac{v^{(s)}_{td+s}}{\max(\eta_{td+s},1)}.
    \label{eq:cum_xtd_app}
\end{equation}
Using~\eqref{eq:cum_xtd_app} and the Lipschitz continuity of $f$ in AS.4, and $\varsigma \leq \max(\eta_{td+s},1)$ for every $t$ and $\ell$, 
\begin{equation}
\begin{aligned}
     \lVert v_{td}^{(\ell)} \rVert^2 & \leq  \left( \lVert v_{td}^{(\ell)} -  v_{td+\ell}^{(\ell)}\rVert + \lVert v_{td+\ell}^{(\ell)} \rVert \right)^2 \leq 2  \lVert v_{td}^{(\ell)} -  v_{td+\ell}^{(\ell)} \rVert^2 + 2\lVert v_{td+\ell}^{(\ell)} \rVert^2\\
     &\leq  2  \lVert v_{td} -  v_{td+\ell}\rVert^2 + 2\lVert v_{td+\ell}^{(\ell)} \rVert^2 \\
     &= 2  \lVert -x_{td}+x_{td+\ell} + P_{C}(x_{td} -g_{td}) -  P_{C}(x_{td+\ell} -g_{td+\ell})\rVert^2 + 2\lVert v_{td+\ell}^{(\ell)} \rVert^2+2\lVert v_{td+\ell}^{(\ell)} \rVert^2\\
     & \leq 4 \lVert -x_{td}+x_{td+\ell} \rVert^2 + 4 \lVert x_{td+\ell}- x_{td}   +g_{td}-g_{td+\ell}  \rVert^2+2\lVert v_{td+\ell}^{(\ell)} \rVert^2\\
     &\leq 12\lVert -x_{td}+x_{td+\ell} \rVert^2+8L^2\lVert -x_{td}+x_{td+\ell} \rVert^2+2\lVert v_{td+\ell}^{(\ell)} \rVert^2\\
     &\leq (12+ 8L^2) \left\lVert \sum_{s=1}^{\ell-1} \frac{v^{(s)}_{td+s}}{\max(\eta_{i,j},1)} \right\rVert^2 + 2\lVert v_{td+\ell}^{(\ell)} \rVert^2\\
     &\leq (12+ 8L^2) \sum_{s=1}^{\ell-1} \left\lVert \frac{v^{(s)}_{td+s}}{\max(\eta_{i,j},1)} \right\rVert^2 + 2\lVert v_{td+\ell}^{(\ell)} \rVert^2\\
     & \leq \frac{(12+ 8L^2)}{\varsigma^2} \sum_{s=1}^{\ell-1} \lVert v^{(s)}_{td+s} \rVert^2 + 2\lVert v_{td+\ell}^{(\ell)} \rVert^2 
\end{aligned}
\label{eq:cyclic_eq_b}
\end{equation}

By applying the same reasoning presented in~\eqref{eq:final_cyclic}, and~\eqref{eq:bound_cyclic} we get 

\begin{equation}
   \frac{1}{T+1} \sum_{t=0}^T \lVert v_{td} \rVert \leq \frac{1}{\sqrt{T+1}} \sqrt{\sum_{t=0}^T \lVert v_{td} \rVert^2} \leq \frac{1}{\sqrt{T+1}} \sqrt{2  \left(  1 + d\frac{6+ 4L^2}{\varsigma^2} \right) \sum_{j=0}^k \lVert v_{j}^{(\ell_j)} \rVert^2}.
   \label{eq:bound_cyclic_b}
\end{equation}
}

\subsection{Supplementary material to Section~\ref{subsec:logistic}}
We report in Table~\ref{tab:res_comp} and Table~\ref{tab:acc_comp} the final gradient norm and the accuracy on the test set, respectively, on the experiment following the same settings described in Subsection~\ref{subsec:logistic}. 
\begin{table}[h!]
\centering
\begin{tabular}{ c || c | c | c | c | c | c | c || c } 
 \toprule 
    &  \multirow{2}{*}{}&  \multicolumn{3}{c|}{\name} &  \multicolumn{3}{c||}{BCG-LS}\\
\hline
\textbf{Data set} & $\delta$ &  \hspace{-0.1cm}GS\hspace{-0.3cm}&\hspace{-0.1cm} UR\hspace{-0.3cm}&\hspace{-0.1cm}Cyclic\hspace{-0.3cm}&\hspace{-0.1cm}GS\hspace{-0.3cm}& \hspace{-0.1cm}UR\hspace{-0.3cm}&\hspace{-0.1cm}Cyclic\hspace{-0.3cm} \\
  \midrule
    \hspace{-0.10cm} Mush  \multirow{5}{*}{}& 0 & $7.3 \cdot 10^{-4}$   &  $\mathbf{4.5 \cdot 10^{-4}}$  & $4.8 \cdot 10^{-4}$ &   0.0011 &   0.0012  &  0.0011    \\
 &   0.10 & $7.2 \cdot 10^{-4}$ &  $\mathbf{4.8 \cdot 10^{-4}}$  & $5.0 \cdot 10^{-4}$ &  0.0513  &	 0.0740 & 0.1297    \\
  &  0.20  & $7.3 \cdot 10^{-4}$ &   $\mathbf{4.9 \cdot 10^{-4}}$  &  $5.1 \cdot 10^{-4}$ &  0.1657 &	0.0125 &  0.0122  \\
  &  0.30 & $ 6.6 \cdot 10^{-4}$  &  $\mathbf{4.8 \cdot 10^{-4}}$  & $4.9 \cdot 10^{-4}$ &  0.1112 &	0.2223 &  0.4120  \\
   \hline
\hspace{-0.10cm} Gisette  \multirow{5}{*}{}& 0 & 0.0144   &  0.0508 &    0.0207&   0.0129 &  0.0151   & \textbf{0.0088}    \\
 &   0.10 & \textbf{0.0063}  &  0.0907  & 0.0279  & 0.0819 &	 0.2145 &  0.1613  \\
  &  0.20  & \textbf{0.0074}  & 0.0712 & 0.0359 & 0.2790 &	  0.2638 &  0.2408  \\
  &  0.30 & \textbf{0.0139} & 0.0762  & 0.0313 & 0.2344   &	  0.8555 &   0.9220 \\
   \hline
\hspace{-0.10cm} MNIST  \multirow{5}{*}{}& 0 & 0.0162   &$\mathbf{4.9 \cdot 10^{-5}}$   &   0.0035  &$ 1.4 \cdot 10^{-4} $&   $ 2.3 \cdot 10^{-4} $  &   0.0045   \\
 &   0.10 & $\mathbf{3.2 \cdot 10^{-5}}$ &  0.0056  & 0.0031 & 0.1285 &	0.2360 & 0.2985   \\
  &  0.20  & 0.0117 &  $\mathbf{5.6 \cdot 10^{-4}}$  & 0.0024 &  0.2282 &	0.2769 & 0.2498   \\
  &  0.30 &  $\mathbf{2.1 \cdot 10^{-6}}$ &  $ 8.2 \cdot 10^{-4} $  & 0.0028 &  0.5221 &	1.5072 & 1.8589   \\
   \hline
   \hspace{-0.10cm} Regedo  \multirow{5}{*}{}& 0 & 0.0639   & $ \mathbf{6.3 \cdot 10^{-4}}$   &   0.0590  &   0.0113 &   0.0084 &   0.0111  \\
 &   0.10 &0.0041   & \textbf{0.0015}   &   0.0473  &   0.3390 &   0.4347  &   0.6469  \\
  &  0.20  & 0.0089   & \textbf{0.0020}   &   0.0521  &   0.3038 &   1.0931  &   0.6725  \\
  &  0.30 & 0.0262   & \textbf{0.0011}   &   0.0468  &   1.0205 &   2.5427  &   2.3687  \\
   \hline
   \hspace{-0.10cm} A9A  \multirow{5}{*}{}& 0 &  $\mathbf{ 9.9 \cdot 10^{-10}}$ & $\mathbf{ 9.9 \cdot 10^{-10}}$   &   $\mathbf{ 9.9 \cdot 10^{-10}}$ &    $1.3 \cdot 10^{-6}$  &   $8.1 \cdot 10^{-7}$  &   $5.4 \cdot 10^{-7}$  \\
 &   0.10 &  $\mathbf{ 9.9 \cdot 10^{-10}}$ & $\mathbf{ 9.9 \cdot 10^{-10}}$ &  $\mathbf{ 9.9 \cdot 10^{-10}}$  &  0.0916  &   0.1656  &   0.1554  \\
  &  0.20  &  $\mathbf{ 9.9 \cdot 10^{-10}}$ & $\mathbf{ 9.9 \cdot 10^{-10}}$ &  $\mathbf{ 9.9 \cdot 10^{-10}}$  &  0.1322  &   0.1787  &   0.1602  \\
  &  0.30 &  $\mathbf{ 9.9 \cdot 10^{-10}}$ & $\mathbf{ 9.9 \cdot 10^{-10}}$ &  $\mathbf{ 9.9 \cdot 10^{-10}}$  & 0.2295  &   0.4788 &   0.6092  \\
   \hline
 \bottomrule 
\end{tabular}
\caption{Final gradient norm for logistic regression with LSP, for different levels of noise $\tau$ according to the definition in~\eqref{eq:noise}. Lowest gradient norm per row are highlighted in bold. } 
\label{tab:res_comp}
\end{table}

\begin{table}[h!]
\centering
\begin{tabular}{ c || c | c | c | c | c | c | c || c } 
 \toprule 
    &  \multirow{2}{*}{}&  \multicolumn{3}{c|}{\name} &  \multicolumn{3}{c||}{BCG-LS}\\
\hline
\textbf{Data set} & $\delta$ &  \hspace{-0.1cm}GS\hspace{-0.3cm}&\hspace{-0.1cm} UR\hspace{-0.3cm}&\hspace{-0.1cm}Cyclic\hspace{-0.3cm}&\hspace{-0.1cm}GS\hspace{-0.3cm}& \hspace{-0.1cm}UR\hspace{-0.3cm}&\hspace{-0.1cm}Cyclic\hspace{-0.3cm} \\
  \midrule
    \hspace{-0.10cm} Mush  \multirow{5}{*}{}& 0 & 96.96  & 97.02  & 97.13 &   96.90 &  \textbf{97.58}  &  96.87    \\
 &   0.10 & 96.96 &  97.03  & 97.01 &  \textbf{98.27}  &	 91.03 & 87.82    \\
  &  0.20  & \textbf{98.99} &   97.27  &  97.26 &  97.77 &	90.34 &  85.18  \\
  &  0.30 & 97.00  &  97.22 & \textbf{97.51} &  94.39 &	80.26 &  79.11  \\
   \hline
\hspace{-0.10cm} Gisette  \multirow{5}{*}{}& 0 & 94.93   &  93.63 &    94.47&  \textbf{ 95.10} &  94.93   & 94.97    \\
 &   0.10 & 92.73  &  93.07  & \textbf{94.76}  & 94.50 &	 93.93 &  94.93  \\
  &  0.20  & \textbf{94.80}  & 94.10 & 94.60 & 93.77 &	  92.16 &  92.43  \\
  &  0.30 & \textbf{94.37} & 91.80  & 94.30 & 94.10   & 85.90 &   87.47 \\
   \hline
\hspace{-0.10cm} MNIST  \multirow{5}{*}{}& 0 & 85.67   & 85.67   &   85.67  & 85.67&  85.67  &   \textbf{86.00}   \\
 &   0.10 & \textbf{85.67}   & \textbf{85.67}   &   \textbf{85.67}  & 84.27 &	80.03 & 80.83   \\
  &  0.20  & \textbf{85.67}   & \textbf{85.67}   &   \textbf{85.67}  & 84.83 &	76.13 & 77.33   \\
  &  0.30 & \textbf{85.67}   & \textbf{85.67}   &   \textbf{85.67}  &  80.43 &	62.57 & 70.70   \\
   \hline
   \hspace{-0.10cm} Regedo  \multirow{5}{*}{}& 0 & 95.13   & \textbf{95.40}   &  \textbf{95.40}  &   94.07 &   94.67 &   94.13  \\
 &   0.10 &95.47   & 95.13   &   \textbf{95.53}  &   94.00 &   94.67  &   93.80  \\
  &  0.20  & 95.40   & 95.33   &   \textbf{95.73}  &   93.27 &   84.13  &   91.13  \\
  &  0.30 & 95.07   & 95.07   &   \textbf{95.47}  &   93.80 &   72.93  &   88.13  \\
   \hline
   \hspace{-0.10cm} A9A  \multirow{5}{*}{}& 0 &  \textbf{84.67} & \textbf{84.67}   &   \textbf{84.67} &  \textbf{84.67} &  \textbf{84.67}  &  \textbf{84.67}  \\
 &   0.10 & \textbf{84.67} & \textbf{84.67} &  \textbf{84.67}  &  84.56  &   74.10  &   76.23  \\
  &  0.20  &  \textbf{84.67} & \textbf{84.67} &  \textbf{84.67}  &  80.63 &   75.90  &   77.53  \\
  &  0.30 & \textbf{84.67}& \textbf{84.67} &  \textbf{84.67}  & 78.50  &   61.07 &   63.70  \\
   \hline
 \bottomrule 
\end{tabular}
\caption{Percentage of correct classification on the testing set, for different levels of noise $\tau$ according to the definition in~\eqref{eq:noise}. Highest percentages per row are highlighted in bold. } 
\label{tab:acc_comp}
\end{table}    
\end{document}